\documentclass[Svgc_modified]{Svgc_modified}
\usepackage{amsmath}
\usepackage{graphicx}
\usepackage{enumitem}
\usepackage{subfigure}
\def\1{$\ }
\def\2{\ $}

\def\N{\mbox{\makebox[.2em][l]{I}N}}

\def\cC{{\cal C}}
\def\cE{{\cal E}}
\def\cH{{\cal H}}

\def\hqed{\hfill\framebox(6,6){\ }}

\def\ul{\underline}

\def\eiE{\cE^{EI}}

\def\tj{\widetilde{j}}
\begin{document}
\title{Cycles as edge intersection hypergraphs}
%\subtitle{Insert your subtitle here, if you have.}
\author{Martin Sonntag\inst{1} \and Hanns-Martin Teichert\inst{2}}% etc
%}                     % Do not remove

%running header
%\titlerunning{Running title}%
%\authorrunning{1st. author\inst{1} \and 2nd. author\inst{2}}%

\institute{Faculty of Mathematics and Computer Science, Technische Universit\"{a}t Bergakademie Freiberg, Pr\"{u}ferstra\ss e 1, 09596 Freiberg, Germany
%\email {teichert@math.uni-luebeck.de}
\and Institute of Mathematics, University of L\"{u}beck, Ratzeburger Allee 160, 23562 L\"{u}beck, Germany
%\email{sonntag@tu-freiberg.de}
}
\maketitle
\begin{abstract}
If $\cH=(V,\cE)$ is a hypergraph, its {\it edge intersection hypergraph} $EI(\cH)=(V,\eiE)$ has the edge set $\eiE=\{e_1 \cap e_2 \ |\ e_1, e_2 \in \cE \ \wedge \ e_1 \neq e_2  \ \wedge  \ |e_1 \cap e_2 |\geq2\}$. Picking up a problem from \cite{ST2019}, for $n \ge 24$ we prove that there is a 3-regular (and - if $n$ is even - 6-uniform) hypergraph $\cH=(V,\cE)$ with $\lceil \frac{n}{2} \rceil$ hyperedges and $EI(\cH) = C_n$.
\end{abstract}
\begin{keyword}
Edge intersection hypergraph
\end{keyword}
\small{\bf  Mathematics Subject Classification 2010:} 05C65
%
%\receive{xxxxxxxxxxxxxxxxxxxx}
%\finalreceive{yyyyyyyyyyyyyyyyy}\
\section{Introduction and basic definitions}
All hypergraphs $\cH = (V(\cH), \cE(\cH))$ and (undirected) graphs $G=(V(G), E(G))$  considered in the following may have isolated vertices but no multiple edges or loops.

A hypergraph $\cH = (V, \cE)$ is {\em $k$-uniform} if all hyperedges $e \in \cE$ have the cardinality $k$.
Trivially, any 2-uniform hypergraph $\cH$ is a graph.
The {\em degree} $d(v)$ of a vertex $v \in V$ is the number of hyperedges $e \in \cE$ being incident to the vertex $v$.
$\cH$ is {\em $r$-regular} if all vertices $v \in V$ have the same degree $r = d(v)$.

If $\cH=(V,\cE)$ is a hypergraph, its {\it edge intersection hypergraph} $EI(\cH)=(V,\eiE)$ has the edge set $\eiE=\{e_1 \cap e_2 \ |\ e_1, e_2 \in \cE \ \wedge \ e_1 \neq e_2  \ \wedge \ |e_1 \cap e_2 |\geq2\}$.
%For $k \ge 2$, the $k$-th iteration of the EI-operator is defined to be $EI^k(\cH):= EI(EI^{k-1}(\cH))$, where  $EI^1(\cH):= EI(\cH)$. Moreover, the EI-number $\kEI(\cH)$ is the smallest $k \in \N$ such that $\cE(EI^k(\cH)) = \emptyset.$

Let $e = \{ v_1, v_2, \ldots, v_l \} \in \eiE$ be a hyperedge in $EI(\cH)$. By definition, in $\cH$ there exist (at least) two hyperedges $e_1, e_2 \in \cE(\cH)$ both containing all the vertices $v_1, v_2, \ldots, v_l$, more precisely $\{ v_1, v_2, \ldots, v_l \} = e_1 \cap e_2 $. In this sense, the hyperedges of $EI(\cH)$ describe sets $\{ v_1, v_2, \ldots, v_l \}$ of vertices having a certain, "strong" neighborhood relation in the original hypergraph $\cH$.

For an application see \cite{ST2019}.
%As an application, we consider a hypergraph $\cH = ( V, \cE)$ representing a communication system. The vertices $v_1, v_2, \ldots, v_n \in V$ and the hyperedges $e_1, e_2, \ldots, e_m \in \cE$ correspond to $n$ people and to $m$ (independent) communication channels, respectively. A group $\{ v_{i_1}, v_{i_2}, \ldots, v_{i_k} \} \subseteq V$ of people can communicate in a conference call if and only if their members use one and the same communication channel, i.e.  there is a hyperedge $e \in \cE$ such that $\{ v_{i_1}, v_{i_2}, \ldots, v_{i_k} \} \subseteq e$.
%If we ask whether or not $v_{i_1}, v_{i_2}, \ldots, v_{i_k} $ can even communicate in a conference call after the breakdown of an arbitrarily chosen communication channel, then this question is equivalent to the problem of the existence of a hyperedge $\eie \in \eiE$ in the edge intersection hypergraph $EI(\cH)$ containing all these vertices, i.e.  $\{ v_{i_1}, v_{i_2}, \ldots, v_{i_k} \} \subseteq \eie$.

Note that there is a significant difference to the well-known notions of the {\em intersection graph} (cf. \cite{Nai}) or {\em edge intersection graph} (cf. \cite{Sku}) $G=(V(G), E(G))$  {\em of linear hypergraphs} $\cH = (V(\cH), \cE(\cH))$, since there we have $V(G) = \cE(\cH)$.

In \cite{Gol}, \cite{Bie} and  \cite{Cam} the same notation is used for so-called {\em edge intersection graphs of paths}, but there the authors consider paths in a given graph $G$ and the vertices of the resulting edge intersection graph correspond to these  paths in the original graph $G$.

Obviously, for certain hypergraphs $\cH$ the edge intersection hypergraph $EI(\cH)$  can be 2-uniform; in this case $EI(\cH)$ is a simple, undirected graph $G$. But in contrast to the intersection graphs or  edge intersection graphs mentioned above,   $G=EI(\cH)$ and $\cH$ have one and the same vertex set $V(G) = V(\cH)$.
Therefore we consistently use our notion "edge intersection hypergraph" also when this hypergraph is 2-uniform.

In \cite{ST2019}, we investigated structural properties of edge intersection hypergraphs and proved that all trees but seven exceptional ones are edge intersection hypergraphs of  3-uniform hypergraphs. In the present paper, for the class of cycles we investigate the following natural question.

\medskip
{\bf Problem 1.} (Problem 3 in {\bf \cite{ST2019})}
Let ${\cal G}$ be a class of graphs, $k \ge 3$, $n_0 \in \N^+ $, $n \ge n_0$ and $G_n \in {\cal G}$ a graph with $n$ vertices. What is the minimum cardinality $| \cE |$ of the edge set of a $k$-uniform hypergraph $\cH_n =(V, \cE)$ with $EI( \cH_n) = G_n$?

\medskip
In order to attack Problem 1, at first it makes sense to investigate simple classes of graphs.
Let us consider $G = C_n$, where $C_n$ denotes the cycle with $n$ vertices.
At the start, let us omit the restriction to uniform hypergraphs.

Of course, for some small $n$,  hypergraphs $\cH=(V, \cE)$ with $EI(\cH) = C_n$ and minimum cardinality $| \cE |$ can be easily found. So by distinction of cases it can be proved that for $n \in \{3,4 \}$ this minimum cardinality is $n+1=4$, whereas for $n \in \{5,6 \}$ the wanted minimum is $n$ (see the strong 3-uniform hypercycle $\hat{\cC}_n^3= (V, \cE)$ with $V=\{ 1,2, \ldots, n \}$ and $\cE = \{ \{ i, i+ 1, i+2 \} \, | \, i \in V \}$ (the vertices taken modulo $n$)).

%= (V, \cE)$ with $V=\{ 1,2, \ldots, n \}$ and $\cE = \{ \{ i, i+ 1, i+2 \} \, | \, i \in V \}$ (the vertices taken modulo $n$) are edge-minimal 3-uniform hypergraphs with $EI(\cH_n^3) = C_n$; this answers the question for $k = 1$ and $n_0 = 4$.

This situation completely changes for larger $n$. Then it seems to be difficult to determine this minimum cardinality $| \cE |$ without additional restrictions on  the  hypergraphs $\cH$ being under consideration.
So in the range $7 \le n \le 23$ only unsatisfying, partial results are known.

This way, $k$-uniformity comes into the play. A first result is the following.

\bigskip
{\bf Corollary 1. (\cite{ST2019})}
For $n \ge 5$ the cycle $C_n$ is an edge intersection hypergraph of a 3-uniform \\[0.3ex] hypergraph, namely $ C_n = EI(\hat{\cC}_n^3)$.

\bigskip
\noindent
Note that $\hat{\cC}_n^3$ is edge minimal in the 3-uniform case.

Considering $k$-uniform hypergraphs ($k \ge 3$), we will see later that multiples of 3   are good candidates for the number $k$. The reason is that we will build the hyperedges $e \in \cE$ of $\cH$ by combining so-called {\em 3-sections} of the vertices of $V= V(C_n)= V( \cH)$ (see Subsection 2.1 for the definition of a $k$-section and Subsection 2.2 for the construction of the hyperedges). Therefore, in our main result, the number $k$ will be chosen equal to 6.
Under the restriction of 6-uniformity, for even $n \ge 24$ we will construct a family of  3-regular hypergraphs $\cH = (V, \cE)$ with $EI(\cH) = C_n$ and minimum cardinality $| \cE | = \frac{n}{2}$.

If $n$ is odd, then we additionally need one hyperedge $e$ of cardinality 3; in this case we obtain $| \cE | = \frac{n+1}{2}$.

\section{Generating $C_n$ as an edge intersection hypergraph}

\subsection{A lower bound for the number of hyperedges}
\label{Sec_low}

At first we will give some notations. For this end, let $n \ge 24$ be even and $\cH=(V, \cE)$ a 6-uniform hypergraph with  $EI(\cH) = (V, \eiE ) = C_n$. In detail, let $C_n =(V, E)$, $V= \{ 1,2, \ldots, n \}$ and
$E = \{ \{ 1,2 \}, \{ 2, 3 \}, \ldots, \{ n-1, n \}, \{ n, 1 \} \}$.
In general,  the vertices in $V$ will be always taken modulo $n$.

For $i \in \{ 1,2, \ldots, n \}$ and $e \in \cE$, a sequence $(i, i+1, \ldots, i+k-1)$ with $\{i, i+1, \ldots, i+k-1 \} \subseteq e$, such that $i-1 \notin e$ and $i+k \notin e$, is referred to as a {\em $k$-section } of $e$ on $C_n$.

Any subset $\{ v_1, v_2 , \ldots, v_k \} \subseteq V $ of $k \ge 2$ vertices containing two vertices $v, v'$ with $| v - v'| \ge 2$ is called a {\em chord} of $C_n$. Since $E(C_n) = \eiE$ cannot contain any chord, for any two distinct hyperedges $e, e' \in \cE$ it holds $| e \cap e' | \leq 2$. For the same reason, in case of $| e \cap e' | = 2$ there exists a vertex $i \in V$ with $| e \cap e' | = \{ i, i+1 \}$.

In our first theorem, we will prove that $\frac{n}{2}$ is a lower bound for the cardinality of the edge set $\cE$ of a 6-uniform hypergraph $\cH =(V, \cE)$ with $EI(\cH) = C_n$, for even $n \ge 24$.

The motivation for 24 as a lower bound for $n$ results from the fact that our main theorem in Subsection \ref{Sec_Cn} provides the construction of such hypergraphs $\cH$ for all even $n = | V | \ge 24$.

\begin{theorem}
\label{Tmin} Let $n \ge 24$ be even and $\cH=(V, \cE)$ a 6-uniform hypergraph with  $EI(\cH) = C_n$.
Then $| \cE | \ge \frac{n}{2}$.
\end{theorem}

\begin{proof}
Let $\cH$ fulfil the assumptions of the Theorem and $e, e' \in \cE(\cH)$ with $e \cap e' = \{ i, i+1 \} \in E(C_n)$, where $i \in V$. We say that the hyperedge $e$  {\em half-generates} the edge $\{ i, i+1 \}$ of $C_n$. The term "half-generate" comes from the fact that we always need at least two hyperedges $e \neq e'$ to generate an edge of $C_n$.

First, we discuss the number $k_e := | \{ \{ i, i+1 \} \, | \, i \in V \, \wedge  \ \{i, i+1 \} \subseteq e \} |$ of the edges of $C_n$ being half-generated by the hyperedge $e$. The following values for $k_e$ may occur:

\begin{enumerate}
\item
$k_e = 5$ -- then $e$ has to consist of a 6-section.

\item
$k_e = 4$ -- then we have the following three possibilities:
\begin{enumerate}
\item
$e$ contains a 5-section and a 1-section;
\item
$e$ consists of a 4-section and a 2-section;
\item
$e$ has  two 3-sections.
\end{enumerate}
\item $k_e = 3$ -- again three variants are possible:
\begin{enumerate}
\item
$e$ includes a 4-section and two 1-sections;
\item
$e$ contains a 3-section, a 2-section and a 1-section;
\item
$e$ is composed of three 2-sections.
\end{enumerate}
\item
$k_e = 2$ -- the hyperedge $e$ consists of
\begin{enumerate}
\item
a 3-section and three 1-sections or
\item
two 2-sections  and two 1-sections.
\end{enumerate}
\item $k_e = 1$ -- now $e$ is the union of a 2-section and four 1-sections.
\item  $k_e = 0$ implies that in $e$ we have six 1-sections, i.e., six  vertices being non-adjacent in $C_n$.
\end{enumerate}

\noindent
 Assume $| \cE | < \frac{n}{2}$ and $k_e \le 4$, for all hyperedges $e \in \cE$. This leads to the contradiction
 $$  | E(C_n) | = n \le \frac{1}{2} \sum_{e \in \cE} k_e \le \frac{1}{2} \cdot | \cE | \cdot 4 < n . $$

Therefore, in case of $| \cE | < \frac{n}{2}$ there has to exist at least one hyperedge  $e' \in \cE$ with $k_{e'} = 5$. Let $e' = \{ i, i+1, \ldots, i+5 \}$, where  $i \in V$.

Moreover, for simplicity we label $V$ so that $e' = \{ 1,2, \ldots, 6 \}$ holds.

The "inner $C_n$-edges" of $e'$, i.e. $\{2,3 \}$, $\{ 3,4 \}$ and $\{ 4,5 \}$ arise as intersections $e_1 \cap e'$, $e_2 \cap e'$ and $e_3 \cap e'$ of $e'$  with pairwise distinct hyperedges $e_1, e_2, e_3 \in \cE \setminus \{ e' \}$. For this reason, each of the three hyperedges     $e_1$, $e_2$ and  $e_3$ has to possess at least one 2-section.

Hence, for $i \in \{ 1,2, 3 \}$ we have $k_{e_i} \le 4$ and we obtain $k_{e'} + \sum_{i=1}^{3} k_{e_i} \le 5 + 12 = 17$.
In order to get  $k_{e'} + \sum_{i=1}^{3} k_{e_i}  = 17 > 16 = 4 \cdot | \{ e', e_1, e_2, e_3 \} |$, necessarily each of the hyperedges $e_1$, $e_2$ and  $e_3$ has to include a 4-section.

Considering the three "middle" pairs of the vertices in the 4-sections
$(i, i+1, i+2, i+3)$, $(j, j+1, j+2, j+3)$ and $(k, k+1, k+2, k+3)$ of $e_1$, $ e_2$ and  $e_3$, respectively, an analog argumentation is true: again we need three hyperedges $e'_1$, $ e'_2$ and  $e'_3$, each of them consisting of a 2-section ($(i+1, i+2)$, $(j+1, j+2)$ and $(k+1, k+2)$, respectively) and a 4-section, to half-generate now $29 > 28 = 4 \cdot | \{ e',e_1, e_2, e_3, e'_1, e'_2, e'_3 \} |$ edges of $C_n$.
%Since $\cH$ (as well as $C_n$) is finite
Because $\cH$ is finite, this leads inductively to the contradiction that there is a hyperedge $e^{\ast} \in \cE$ containing at least one 2-section but no 4-section. Then $e^{\ast}$ half-generates at most 3 edges of $C_n$. Let $e', e_1, e_2, e_3, e'_1, e'_2, e'_3, \ldots, e^{\ast}$ be the set of all hyperedges used up to this point and $t$ be the number of these hyperedges. We easily see
$$ k_{e'} + k_{e_1} + k_{e_2} + k_{e_3} + k_{e'_1} + k_{e'_2} + k_{e'_3} + \ldots + k_{e^{\ast}} = 5 + k_{e_1} + k_{e_2} + k_{e_3} + k_{e'_1} + k_{e'_2} + k_{e'_3} + \ldots + 3 \le 4 t.$$

This argument is valid for each hyperedge $e' \in \cE$ with $k_{e'} = 5$.  Moreover, for all other hyperedges $\tilde{e} \in \cE$ we know $k_{\tilde{e}} \le 4$ and this leads to $\sum_{e \in \cE} k_e \le 4 \cdot | \cE |$. This yields the same contradiction as above, namely

 $$  n =  | E(C_n) |  \le \frac{1}{2} \sum_{e \in \cE} k_e \le 2 \cdot | \cE |  < 2 \cdot \frac{n}{2}  $$
  and the proof is complete.\hqed
  %that in order to generate $C_n$ as the edge intersection hypergraph of a 6-uniform hypergraph $\cH$ at least $\frac{n}{2}$ hyperedges have to exist in $\cH$.

\end{proof}

\subsection{The construction of hypergraphs $\cH$ with $EI(\cH) = C_n$}
\label{Sec_Cn}

Our main result  is the following.

\begin{theorem}
\label{Teven} Let $n \ge 24$. Then there exists a hypergraph $\cH=(V, \cE)$  with  $EI(\cH) = C_n$ such that the following holds.
\begin{enumerate}
\item[(i)]
If $n$ is even, then $\cH$ is 3-regul\"{a}r, 6-uniform and $|\cE| = \frac{n}{2}$.
\item[(ii)]
If $n$ is odd, then $\cH$ is 3-regul\"{a}r, $|\cE| = \frac{n+1}{2}$, $\cH$ contains one hyperedge
%${\underline e}$
of cardinality 3 and all other hyperedges in $\cH$ have cardinality 6.
\end{enumerate}
\end{theorem}

Note that the lower bound $\frac{n}{2}$ for the cardinality $| \cE|$  given in Theorem \ref{Tmin} is sharp due to Theorem \ref{Teven}(i).

Depending on $n$, the verification of Theorem \ref{Teven} will be done by proving the following lemmata.

\begin{lemma}
\label{Leven}
Let $k, l, n \in \N$ with $k \ge 3, k \neq 4, l \in \{ 0,2,4,6 \}$ and $n = 8 k + l$. \\
Then there exists a 3-regular, 6-uniform hypergraph $\cH=(V, \cE)$  with  $EI(\cH) = C_n$ and $| \cE | = \frac{n}{2}$.
\end{lemma}

\begin{lemma}
\label{Lk4}
Let $l, n \in \N$ with $l \in \{ 0,2,4,6 \}$ and $n = 32 + l$. \\
Then there exists a 3-regular, 6-uniform hypergraph $\cH=(V, \cE)$  with  $EI(\cH) = C_n$ and $| \cE | = \frac{n}{2}$.
\end{lemma}

\begin{lemma}
\label{Lodd}
Let $n \in \N$ be odd. \\
Then there exists a 3-regular hypergraph $\cH$ with  $EI(\cH) = C_n$, $|\cE| = \frac{n+1}{2}$, $\cH$ contains one hyperedge
%${\underline e}$
of cardinality 3 and all other hyperedges in $\cH$ have cardinality 6.
\end{lemma}

Lemma 1, 2 and 3  will be shown separately in Subsection 2.2.1, 2.2.2 and 2.2.3, respectively. In every case, the proof requires the following steps.

\begin{description}
\item[\bf\ul{Step 1.}] Construct the hyperedges of $\cH = (V, \cE(\cH))$.
\item[\bf\ul{Step 2.}]  Verify that the hyperedges of $\cH$ generate all edges of $C_n$: $E(C_n) \subseteq \cE( EI(\cH))$.
\item[\bf\ul{Step 3.}]  Verify that the hyperedges of $\cH$ do not generate any chord in $C_n$: $\cE( EI(\cH)) \subseteq E(C_n)$.
\end{description}

\medskip

\subsubsection{Proof of Lemma 1}
\label{Sec_L1}
$\/$

\bigskip
\noindent
{\bf\ul{Step 1.}} Construction of the set of hyperedges $\cE(\cH)$ of the hypergraph $\cH$.

\medskip
At first we give a rough description of the construction principle for the hyperedges.

\medskip
In the {\ul{\em basic construction}} (this corresponds to  $l=0$) for $j \in \{1,5,9, \ldots, \frac{n}{2} - 3 \}$ and $\widetilde{j} = \frac{j-1}{4} \in \{ 0,1, \ldots, \frac{n}{8} - 1 \}$ we form so-called {\em $4$-groups}
$G_{\widetilde{j}} = \{ e_j, e_{j+1}, e_{j+2}, e_{j+3} \}$ of hyperedges. Each of the constructed hyperedges consists of two 3-sections on $C_n$, in detail the hyperedges will have the structure  $e = \{ p', p'+1, p'+2, q', q'+1, q'+2 \}$ with $p', q' \in V$ and $| q' - p' | \ge 6$. Note that we take the vertices of $V = \{ 1,2, \ldots, n \}$ modulo $n$, therefore $| q' - p' | \ge 6$ is meant in the sense that the distance between $p'$ and $q'$ on the cycle $C_n$ is at least 6.

For every $j \in  \{1,5,9, \ldots, \frac{n}{2} - 3 \}$ the so-called {\em first $3$-sections} of the hyperedges $ e_j, e_{j+1}, e_{j+2}, e_{j+3} $ in such a 4-group overlap each other in the following way: \\[0.5ex]
$e_{j \phantom{+1}} = \{ p, p+1, p+2, \ldots \}$,\\
$e_{j+1} = \{ p+1, p+2, p+3, \ldots \}$,\\
$e_{j+2} = \{ p+2, p+3, p+4, \ldots \}$,\\
$e_{j+3} = \{ p+3, p+4, p+5, \ldots \}$, for certain $p \in V$.

\smallskip
This property of overlapping (by two vertices, if we consider $e_k$ and $e_{k+1}$ ($k \in \{j, j+1, j+2\}$)) determines the first 3-sections of the hyperedges uniquely, since
the other 3-sections  do not overlap. These other 3-sections are referred  to as the {\em second $3$-sections} of the hyperedges.
Considering those  $3$-sections, we will see that no hyperedge $e \in G_{\widetilde{j}}$ has a non-empty intersection with a second $3$-section of any of the other hyperedges $e' \in G_{\widetilde{j}} \setminus \{ e \}$.

\medskip
In the  {\ul{\em supplemental construction}} (corresponding to $l \in \{ 2,4,6 \}$), we replace  $\frac{l}{2}$ of the 4-groups $G_{\widetilde{j}}$   by {\em $5$-groups} $G_{\widetilde{j}} = \{ e_j, e_{j+1}, e_{j+2},$ $ e_{j+3}, e_{j+4} \}$ of hyperedges. In comparison with a 4-group $G_{\widetilde{j}} = \{ e_j, e_{j+1}, e_{j+2}, e_{j+3} \}$  in the basic construction (see above) we add  yet another hyperedge $e_{j+4} = \{ p'', p''+1, p''+2, q'', q''+1, q''+2 \}$ in order to obtain the needed 5-group. Looking at the detailed definitions of the hyperedges, later we will see that most of the properties described above for the hyperedges in the 4-groups will be preserved for the hyperedges in the new 5-groups, only little modifications will appear.

So $e_{j+4}$ will continue the overlapping of the first 3-sections; using the number $p$ from above we have  $e_{j+4} = \{ p+4, p+5, p+6, \ldots \}$. For $e = \{ p', p'+1, p'+2, q', q'+1, q'+2 \} \in \{ e_j, e_{j+1}, e_{j+2}, e_{j+3} \}$ the inequality $| q' - p' | \ge 6$ remains valid; for $e_{j+4} = \{ p'', p''+1, p''+2, q'', q''+1, q''+2 \}$ we obtain $| q'' - p'' | \ge 5$. The non-overlapping property of the second 3-sections of $e \in \{ e_j, e_{j+1}, e_{j+2}, e_{j+3} \}$ and the other hyperedges of the 4-group (see above) is preserved also for the new 5-groups $G_{\widetilde{j}} = \{ e_j, e_{j+1}, e_{j+2}, e_{j+3}, e_{j+4} \}$, only for $e_{j+4} = \{ p'', p''+1, p''+2, q'', q''+1, q''+2 \}$, $e_{j} = \{ p, p+1, p+2, \ldots \}$  and $e_{j+1} = \{ p+1, p+2, p+3, \ldots \}$ we have $e_{j+4} \cap e_{j} = \{ p = q'' + 1, p + 1 = q'' + 2 \}$ and $e_{j+4} \cap e_{j + 1} = \{ p + 1 = q'' + 2 \}$.

\bigskip
Now we give the detailed definitions of the hyperedges.
We begin with the $\frac{l}{2}$ 5-groups of hyperedges, where $ l \in \{ 0,2,4,6 \}$.

\medskip

\noindent
\underline{(I): $\tj \in \{ 0,1, \ldots, \frac{l}{2} - 2 \}$ \; or \; $n=30 \, \wedge \, \tj = \frac{l}{2} - 1.$}

\smallskip
Let $j = 5 \tj + 1$, i.e. $j \in \{ 1,6,11 \}$ (note that $j=11$ is possible only for $n=30$). \\[1ex]
$e_{j \phantom{+1}}= \{ 10 \tj+1, 10 \tj+2, 10 \tj+3, 10 \tj+8, 10 \tj+9, 10 \tj+10 \}, $\\[0.5ex]
$e_{j+1}= \{ 10 \tj+2, 10 \tj+3, 10 \tj+4, 10 \tj+19, 10 \tj+20, 10 \tj+21 \}, $\\[0.5ex]
$e_{j+2}= \{ 10 \tj+3, 10 \tj+4, 10 \tj+5, 10 \tj+16, 10 \tj+17, 10 \tj+18 \}, $\\[0.5ex]
$e_{j+3}= \{ 10 \tj+4, 10 \tj+5, 10 \tj+6, 10 \tj-3, 10 \tj-2, 10 \tj-1 \}, $\\[0.5ex]
$e_{j+4}= \{ 10 \tj+5, 10 \tj+6, 10 \tj+7, 10 \tj, 10 \tj+1, 10 \tj+2 \}. $

\medskip
Whereas in (I) we have 0, 1 or 2 such 5-groups $G_{\widetilde{j}}$ (depending on $l \in \{ 0,2,4 \}$) --  with the exception $n=30$, where three 5-groups can occur -- the following case (II) describes only one 5-group, namely the {\em largest} 5-group $G_{\frac{l}{2} - 1}$, which is the 5-group with the largest index $\widetilde{j} = \frac{l}{2} - 1 $.

\bigskip

\noindent
\underline{(II): $n \neq 30 \, \wedge \, l \ge 2 \, \wedge \, \tj = \frac{l}{2} - 1.$}

\smallskip
 $j = 5 \tj + 1$, $e_j$, $e_{j+3}$ and $e_{j+4}$ are the same as in (I).
 Since the next group  $G_{\frac{l}{2}}$ after  $G_{\frac{l}{2} - 1}$ is a 4-group, we have to modify $e_{j+1}$ and $e_{j+2}$ as follows. \\[1ex]
 $e_{j+1}= \{ 10 \tj+2, 10 \tj+3, 10 \tj+4,10 \tj+17, 10 \tj+18, 10 \tj+19 \}, $\\[0.5ex]
$e_{j+2}= \{ 10 \tj+3, 10 \tj+4, 10 \tj+5, 10 \tj+14, 10 \tj+15, 10 \tj+16 \}.$

\bigskip
Because of $| \cE | = \frac{n}{2}$ and  since we have defined $\frac{l}{2}$ of such 5-groups, namely\\[1ex]
$ \begin{array}{l@{\, = \,}l}
G_{0} & \{ e_1, e_2, e_3, e_4, e_5 \}, \\
G_{1} & \{ e_6, e_7, e_8, e_9, e_{10} \}, \ldots, \\
G_{\frac{l}{2} - 1} & \{ e_{\frac{5}{2}l - 4}, e_{\frac{5}{2}l - 3}, e_{\frac{5}{2}l - 2}, e_{\frac{5}{2}l - 1}, e_{\frac{5}{2}l} \},
\end{array} $ \\[1ex]
we need $\frac{1}{4} ( \frac{n}{2} - \frac{5}{2}l) = \frac{1}{8} (n - 5 \, l)$ 4-groups, in detail\\[1ex]
$ \begin{array}{l@{\, = \,}l}
G_{\frac{l}{2}} & \{ e_{\frac{5}{2}l + 1}, e_{\frac{5}{2}l + 2}, e_{\frac{5}{2}l + 3}, e_{\frac{5}{2}l + 4} \}, \\
G_{\frac{l}{2} + 1 } & \{ e_{\frac{5}{2}l + 5}, e_{\frac{5}{2}l + 6}, e_{\frac{5}{2}l + 7}, e_{\frac{5}{2}l + 8} \}, \ldots, \\
G_{\frac{n-l}{8} - 1} & \{ e_{ \frac{n}{2} - 3}, e_{\frac{n}{2} - 2}, e_{\frac{n}{2} - 1}, e_{\frac{n}{2}} \}.
\end{array} $ \\

\smallskip

Here are the 4-groups.

\bigskip

\noindent
\underline{(III): $\tj \in \{ \frac{l}{2}, \frac{l}{2} + 1 \ldots, \frac{n-l}{8} - 2\}$ \; or \; $  l=0   \, \wedge \,  \tj = \frac{n-l}{8} - 1.$}

\smallskip
Let $j = 4 \tj + \frac{l}{2} + 1$, i.e. $j \in \{ \frac{5}{2}l + 1, \frac{5}{2}l + 5, \ldots, \frac{n}{2} - 7, \frac{n}{2} - 3 \}$ (note that $j=  \frac{n}{2} - 3 $ is possible only for $l=0$).

With \\[-2ex]
\hspace*{5cm} $x = \left\{ \begin{array}{l@{\quad,\quad}l}
 8 \tj & l=0 \\[0.8ex] 8 \tj + l - 1 & l>0
 \end{array} \right.$ \\

 we define the hyperedges of $G_{\tj}$. \\[1ex]
 $e_{j \phantom{+1}}= \{ x+1, x+2, x+3, x+7, x+8, x+9 \}, $\\[0.5ex]
$e_{j+1}= \{ x+2, x+3, x+4, x+16, x+17, x+18 \}, $\\[0.5ex]
$e_{j+2}= \{ x+3, x+4, x+5, x+13, x+14, x+15 \}, $\\[0.5ex]
$e_{j+3}= \{ x+4, x+5, x+6, x-2, x-1, x \}. $

\bigskip

\noindent
\underline{(IV): $l \ge 2 \, \wedge \,  \tj = \frac{n-l}{8} - 1.$}

\smallskip
In comparison with (III) we have to modify only the second and the third hyperedge; we set\\[1ex]
$e_{j+1}= \{ x+2, x+3, x+4, x+18, x+19, x+20 \} $ \quad and\\[0.5ex]
$e_{j+2}= \{ x+3, x+4, x+5, x+15, x+16, x+17 \}. $

\medskip

Owing to  $j = \frac{n}{2} - 3$ and $ x = 8 \tj + l - 1 = n - 9$  we obtain finally\\[1ex]
 $e_{\frac{n}{2} - 3}= \{ n-8,n-7, n-6, n-2,n-1,n \}, $\\[0.5ex]
$e_{\frac{n}{2} - 2}= \{ n-7,n-6, n-5, 9,10,11 \}, $\\[0.5ex]
$e_{\frac{n}{2} - 1}= \{ n-6,n-5, n-4, 6,7,8 \}, $\\[0.5ex]
$e_{\frac{n}{2} \phantom{ - 1}}= \{ n-5,n-4, n-3, n-11,n-10,n-9  \}. $

\bigskip
Thus the construction of the hyperedges of $\cH$ is complete.

\medskip
In the construction of $\cE(\cH)$ we  gave the definitions of the hyperedges in a -- more or less -- formal and compact form.
In order to verify that $EI(\cH)$ contains all edges of $C_n$ (Step 2) but no chords (Step 3), it is more favorable to deviate from the above and handle the cases $l = 0 $ and $l \in \{ 2,4,6 \}$ separately. \\[1ex]

%In order to attack Step 2 and Step 3, we deviate from the above and handle the cases $l = 0 $ and $l \in \{ 2,4,6 \}$ separately. The reason is that in the construction of $\cE(\cH)$ we only gave the definitions of the hyperedges in a formal and compact form. Now we have to verify that $EI(\cH)$ contains all edges of $C_n$ (Step 2) but no chords (Step 3). This is much easier by treating not so complex but rather simpler cases what we will do here. \\[1ex]

\paragraph{2.2.1.1 Basic construction: $l = 0 $.}
\label{Sec_L11}
$\/$

\medskip
First,  Step 1 (see (I)--(IV) above) can be adapted to the case $l=0$.

\medskip
 \noindent
{\bf\ul{Step 1.}} Construction of the set of hyperedges $\cE(\cH)$ of the hypergraph $\cH$.

\medskip
Let $j \in \{1,5,9, \ldots, \frac{n}{2} - 3 \}$ and $\widetilde{j} = \frac{j-1}{4} \in \{ 0,1, \ldots, \frac{n}{8} - 1 \}$.
Considering the 4-groups

\smallskip
\noindent
$ \begin{array}{l@{\, = \,}l}
G_{0} & \{ e_1, e_2, e_3, e_4 \}, \\
G_{1} & \{ e_5, e_6, e_7, e_8 \}, \ldots, \\
G_{\frac{n}{8} - 1} & \{ e_{ \frac{n}{2} - 3}, e_{\frac{n}{2} - 2}, e_{\frac{n}{2} - 1}, e_{\frac{n}{2}} \},
\end{array} $ \\

the 4-group $G_{\tj}$ contains the hyperedges \\[1ex]
$ \begin{array}{l@{\, = \, \{ \,}l@{, \,}l@{, \,}l@{, \,}l@{, \,}l@{, \,}l@{\, \},}}
e_{j}   & {\bf 8 \tj +1} & {\bf 8 \tj +2} & {\bf 8 \tj +3} &  8 \tj +7  & 8 \tj +8  & 8 \tj +9 \\[0.5ex]
e_{j+1} & {\bf 8 \tj +2} & {\bf 8 \tj +3} & {\bf 8 \tj +4} &  8 \tj +16 & 8 \tj +17 & 8 \tj +18 \\[0.5ex]
e_{j+2} & {\bf 8 \tj +3} & {\bf 8 \tj +4} & {\bf 8 \tj +5} &  8 \tj +13 & 8 \tj +14 & 8 \tj +15 \\[0.5ex]
e_{j+3} & {\bf 8 \tj +4} & {\bf 8 \tj +5} & {\bf 8 \tj +6} &  8 \tj -2  & 8 \tj -1  & 8 \tj
\end{array} $ \\[1ex]

where
$j = 4 \tj  + 1$.

In each hyperedge $e_\tau$, the vertices printed bold are the vertices of the {\em first 3-section} of the hyperedge $e_\tau$ (see  2.2.1).

Remember that we take the numbers  of the vertices $1,2, \ldots, n$  modulo $n$, the indices of the hyperedges $e_1, e_2, \ldots, e_{\frac{n}{2}}$  modulo $\frac{n}{2}$ and the indices of the 4-groups
$ G_{0}, G_{1}, \ldots,
G_{\frac{n}{8} - 1}$ modulo $\frac{n}{8}$.

As an instance for the construction of the hyperedges, we choose $n=24$ and consider the hypergraph $\cH = (V, \{e_1, e_2, \ldots, e_{12} \})$ with $EI(\cH) = C_{24}$. This may also ease the understanding of Step 2 and Step 3  below.

\smallskip
\begin{example}
$\cH = (V, \{e_1, e_2, \ldots, e_{12} \})$ with $EI(\cH) = C_{24}$ has the following hyperedges.\\[1ex]
$ \begin{array}{l@{\, = \, \{ \,}c@{, \,}c@{, \,}c@{, \,}c@{, \,}c@{, \,}c@{\, \},}}
e_{1} & {\bf 1} & {\bf 2} & {\bf 3} &  7  & 8  & 9 \\[0.5ex]
e_{2} & {\bf 2} & {\bf 3} & {\bf 4} & 16 & 17 & 18 \\[0.5ex]
e_{3} & {\bf 3} & {\bf 4} & {\bf 5} &  13 & 14 & 15 \\[0.5ex]
e_{4} & {\bf 4} & {\bf 5} & {\bf 6} &  22  & 23  & 24  \\[0.5ex]
e_{5} & {\bf 9} & {\bf 10} & {\bf 11} &  15  & 16  & 17 \\[0.5ex]
e_{6} & 1 & 2 & {\bf 10} & {\bf 11} & {\bf 12} & 24 \\[0.5ex]
e_{7} & {\bf 11} & {\bf 12} & {\bf 13} &  21 & 22 & 23 \\[0.5ex]
e_{8} & 6 & 7 & 8 &  {\bf 12}  & {\bf 13}  & {\bf 14}  \\[0.5ex]
e_{9} & 1 & {\bf 17} & {\bf 18} &  {\bf 19}  & 23  & 24 \\[0.5ex]
e_{10} & 8 & 9 & 10 & {\bf 18} & {\bf 19} & {\bf 20} \\[0.5ex]
e_{11} & 5 & 6 & 7 &  {\bf 19} & {\bf 20} & {\bf 21} \\[0.5ex]
e_{12} & 14 & 15 & 16 &  {\bf 20}  & {\bf 21}  & {\bf 22}
\end{array} $ \\[1.5ex]
and the edges of $C_{24}$ result from the intersections \\[0.5ex]
$e_1 \cap e_2 = \{2,3\}$, $e_2 \cap e_3 = \{3,4\}$, $e_3 \cap e_4 = \{4,5\}$,  $e_4 \cap e_{11} = \{5,6\}$, \\
 $e_8 \cap e_{11} = \{6,7\}$, $e_1 \cap e_8 = \{7,8\}$, $e_1 \cap e_{10} = \{8,9\}$, $e_5 \cap e_{10} = \{9,10\}$, \\ $e_5 \cap e_6 = \{10,11\}$, $e_{6} \cap e_{7} = \{11,12\}$, $e_7 \cap e_8 = \{12,13\}$, $e_3 \cap e_8 = \{13,14\}$, \\
   $e_3 \cap e_{12} = \{14,15\}$, $e_5 \cap e_{12} = \{15,16\}$,
 $e_2 \cap e_5 = \{16,17\}$, $e_2 \cap e_{9} = \{17,18\}$,\\
  $e_9 \cap e_{10} = \{18,19\}$, $e_{10} \cap e_{11} = \{19,20\}$, $e_{11} \cap e_{12} = \{20,21\}$, $e_7 \cap e_{12} = \{21,22\}$,\\
   $e_4 \cap e_7 = \{22,23\}$,  $e_4 \cap e_{9} = \{23,24\}$, $e_6 \cap e_9 = \{24,1\}$ and $e_1 \cap e_6 = \{1,2\}$.

\smallskip
For $i \neq j$, it can be shown easily that each of the remaining intersections $e_i \cap e_j$ of pairs $e_i, e_j \in \cE(\cH)$ contains less than two vertices and, therefore, the above edges of $C_{24}$ are the only (hyper-)edges in $EI(\cH)$, i.e. there are no chords in $EI(\cH)$ or - by other words - we have $EI(\cH)= C_{24}$.

\smallskip
Alternatively, the computation of the edge set   $\cE(EI(\cH))$ of the edge intersection hypergraph $EI(\cH)$ of a given hypergraph $\cH=(V, \cE)$ can be done   using the computer algebra system MATHEMATICA$^{\textregistered}$ with the function
\\[0.8ex]
\indent
$EEI[eh\_] :=
 Complement[
  Select[Union[Flatten[Outer[Intersection, eh, eh, 1], 1]],$
\vspace{-1ex}
\begin{flushright} $   Length[\#] > 1 \&], eh],$ \end{flushright}

\smallskip
\noindent
where the argument $eh$ has to be the list of the hyperedges of $\cH$ in the form $\{\{a_1,a_2, \ldots, a_{k_a} \}, \ldots,$ $ \{z_1, z_2, \ldots, z_{k_z} \} \}$. Then $EEI[eh]$ provides the list of the hyperedges of $EI(\cH)$.

Note that this function only works correctly if the hypergraph $\cH$ does not contain any hyperedges $e, e'$ with $e \subset e'$.
%Note that this function only works correctly if in $\cE$ there are no hyperedges $e \neq e'$ with $e \subseteq e'$.

\end{example}

\smallskip
Now we show\\[2ex]
{\bf\ul{Step 2.}} {\em   The hyperedges of $\cH$ generate all edges of $C_n$: $E(C_n) \subseteq \cE( EI(\cH))$.}

\medskip
We distinguish several subcases. In each subcase, the given equations are valid for all $\widetilde{j}  \in \{ 0,1, \ldots, \frac{n}{8} - 1 \}$ and $j  = 4 \tj  + 1 \in \{1,5,9, \ldots, \frac{n}{2} - 3 \}$.\\[2ex]

%\pagebreak
\noindent
\ul {\em $\mathbf{(\alpha)}$: Edges of $C_n$ generated only by the first 3-sections of hyperedges.} \\[1.3ex]
$\{ 8 \tj +2, 8 \tj +3 \} \, = \, e_j \, \cap e_{j+1} \, = \, e_{4 \tj  + 1} \, \cap e_{4 \tj  + 2} $,\\
$\{ 8 \tj +3, 8 \tj +4 \} \, = \, e_{j+1} \, \cap e_{j+2}\, = \, e_{4 \tj  + 2} \, \cap e_{4 \tj  + 3}$ \quad and\\
$\{ 8 \tj +4, 8 \tj +5 \} \, = \, e_{j+2} \, \cap e_{j+3}\, = \, e_{4 \tj  + 3} \, \cap e_{4 \tj  + 4}$.\\[2ex]
\ul {\em $\mathbf{(\beta)}$: Edges of $C_n$ generated by a first 3-section and a second 3-section of two hyperedges.} \\[1.3ex]
The edge $\{ 8 \tj +1, 8 \tj +2 \}$ and $\{ 8 \tj +5, 8 \tj +6 \}$ of $C_n$ is contained only in one first 3-section, namely in the
hyperedge $e_j = e_{4 \tj  + 1}$ and $e_{j+3} = e_{4 \tj  + 4}$, respectively.

With $\widetilde{j'} = \tj -2$ and $j' = 4 \widetilde{j'}  + 1 $ we obtain \\[0.5ex]
$\{ 8 \tj +1, 8 \tj +2 \} \, = \, \{ 8 \widetilde{j'} +17, 8 \widetilde{j'} +18 \} \, = \,  e_j \, \cap e_{j'+1} \, = \, e_{4 \tj  + 1} \, \cap e_{4 \widetilde{j'}  + 2}  \, = \, e_{4 \tj  + 1} \, \cap e_{4 \tj -6}$.

Analogously, with  $\widetilde{j'} = \tj -1$ and $j' = 4 \widetilde{j'}  + 1 $ it follows \\[0.5ex]
$\{ 8 \tj +5, 8 \tj +6 \} \, = \, \{ 8 \widetilde{j'} +13, 8 \widetilde{j'} +14 \} \, = \,  e_{j+3} \, \cap e_{j'+2} \, = \, e_{4 \tj  + 4} \, \cap e_{4 \widetilde{j'}  + 3}  \, = \, e_{4 \tj  + 4} \, \cap e_{4 \tj -1}$.\\[2ex]
\ul {\em $\mathbf{(\gamma)}$: Edges of $C_n$ generated only by the second 3-sections of two hyperedges.} \\[1.3ex]
The edges $\{8\widetilde{j} + 1,8\widetilde{j} + 2\}, \{8\widetilde{j} + 2,8\widetilde{j} + 3\}, \{8\widetilde{j} + 3,8\widetilde{j} + 4\},\{8\widetilde{j} + 4,8\widetilde{j} + 5\},\{8\widetilde{j} + 5,8\widetilde{j} + 6\}$ \\[0.5ex]
of $C_n$ have been generated in the subcases {\em $\mathbf{(\alpha)}$} and {\em $\mathbf{(\beta)}$}, for all $\widetilde{j}  \in \{ 0,1, \ldots, \frac{n}{8} - 1 \}$.

To complete the proof of Step 2, now we consider the edges $\{8\widetilde{j} + 6,8\widetilde{j} + 7\},$ \\[0.5ex]$\{8\widetilde{j} + 7,8\widetilde{j} + 8\}, \{8\widetilde{j} + 8,8\widetilde{j} + 9\}$.

\pagebreak
With $\widetilde{j'} = \widetilde{j} -1$ and $\widetilde{j''} = \widetilde{j} + 1 $ we get \\[0.5ex]
$\{ 8 \widetilde{j} + 6, 8 \widetilde{j} + 7 \} \, = \, \{ 8 \widetilde{j'}+14, 8\widetilde{j'} +15 \} \, = \, \{ 8 \widetilde{j''} - 2, 8 \widetilde{j''} - 1 \} \, $ \\[0.5ex]
  $= \, e_{j'+2} \, \cap e_{j''+3} \, = \, e_{4 \widetilde{j'}  + 3} \, \cap e_{4 \widetilde{j''}  + 4}  \, = \, e_{4 \tj  - 1} \, \cap e_{4 \tj + 8}$.

\smallskip
Setting $\widetilde{j'} = \widetilde{j}$ and $\widetilde{j''} = \widetilde{j} + 1 $ we have \\[0.5ex]
$\{ 8 \widetilde{j} + 7, 8 \widetilde{j} + 8 \} \, = \, \{ 8 \widetilde{j'}+7, 8\widetilde{j'} +8 \} \, = \, \{ 8 \widetilde{j''} - 1, 8 \widetilde{j''} \} \, $ \\[0.5ex]
  $= \, e_{j'} \, \cap e_{j''+3} \, = \, e_{4 \widetilde{j'}  + 1} \, \cap e_{4 \widetilde{j''}  + 4}  \, = \, e_{4 \tj  + 1} \, \cap e_{4 \tj + 8}$.

\smallskip
Finally, we choose $\widetilde{j'} = \widetilde{j}$ and $\widetilde{j''} = \widetilde{j} - 1 $. This leads to \\[0.5ex]
$\{ 8 \widetilde{j} + 8, 8 \widetilde{j} + 9 \} \, = \, \{ 8 \widetilde{j'}+8, 8\widetilde{j'} + 9 \} \, = \, \{ 8 \widetilde{j''} + 16, 8 \widetilde{j''} + 17 \} \, $ \\[0.5ex]
  $= \, e_{j'} \, \cap e_{j''+1} \, = \, e_{4 \widetilde{j'}  + 1} \, \cap e_{4 \widetilde{j''}  + 2}  \, = \, e_{4 \tj  + 1} \, \cap e_{4 \tj - 2}$.

\bigskip
With this, we have shown  $E(C_n) \subseteq \cE(EI(\cH))$.\\[3ex]
{\bf\ul{Step 3.}}  The hyperedges of $\cH$ do not generate any chord in $C_n$: $\cE( EI(\cH)) \subseteq E(C_n)$.

\smallskip
\begin{enumerate}
\item[(i)]
 Using the construction of the hyperedges of $\cH$, for all hyperedges $e_f \neq e_g$ we verify that $e_f$ and $e_g$ do never have a 3-section in common. Note that we use the numbering of the hyperedges from 2.2.1, (I)--(IV).

Lets have a look at the remainders modulo 8 of the vertices contained in a 3-section of an arbitrary hyperedge. We find the following sets of remainders for the first 3-sections: $\{1,2,3 \}, \{2,3,4 \}, \{3,4,5 \}, \{ 4,5,6 \}$. For the second 3-sections we have the sets \\ $\{ 7,0,1 \}, \{ 0,1,2 \}, \{5,6,7 \}, \{6,7,0 \}$. Therefore, all these sets of remainders are pairwise distinct.

Assume that $e_f$ and $e_g$ have a 3-section in common. Because of the pairwise distinctness of the sets of the remainders mentioned above, it follows $f \equiv g \; \mbox{mod} \, 4$ and the common 3-section has to be the first or the second 3-section of both,  $e_f$ as well as $e_g$. Looking at the definition of the hyperedges, this leads to $f = g$ in contradiction to  $e_f \neq e_g$.

\medskip
Consequently,  a chord in $EI(\cH)$ cannot be obtained as an intersection of a 3-section $\{x, x+1, x+2 \} \subset e$  and a 3-section $\{y, y+1, y+2 \} \subset e'$ of any two hyperedges $e$ and $e'$ of $\cH$.
\item[(ii)]
If $e = \{ p, p+1, p+2, q, q+1, q+2 \} \in \cE$ is an arbitrary hyperedge, then the construction of $\cH$ provides that the distance $d_{C_n}( p+2,q)$ of the vertices $p+2$ and $q$ along the cycle $C_n$ is at least 4; analogously $d_{C_n}(q+2, p) \ge 4$.
Hence the intersection of $e$ with one 3-section $\{ y, y+1, y+2 \} \subset e'$ of another hyperedge $e'$ can only result in the (welcome) edges $\{p, p+1 \}, \{p+1, p+2 \}, \{q,q+1 \}$ or $\{q+1, q+2 \}$, which are edges of the cycle $C_n$.

\end{enumerate}

Because of (i) and (ii), for any chord $\{k,l \}$ $(|k - l| > 1 \, \wedge \, \{k,l\} \neq \{ n, 1 \})$ in $EI(\cH)$ resulting from the hyperedge $e= \{ p, p+1, p+2, q, q+1, q+2 \}$ and a second hyperedge $e' = \{ r, r+1, r+2, s, s+1, s+2 \}$ one of the following situations must occur. (For our investigations let $\{ p, p+1, p+2 \} $ and $\{ r, r+1, r+2 \}$ be the first 3-section of the hyperedge $e$ and $e'$, respectively.) \\[1.5ex]
(A): $k \in \{ p, p+1, p+2 \} \, \cap \, \{ r, r+1, r+2 \} \; \wedge \; l \in \{ q, q+1, q+2 \} \, \cap \, \{ s, s+1, s+2 \}$.\\[1.5ex]
(B): $k \in \{ p, p+1, p+2 \} \, \cap \, \{ s, s+1, s+2 \} \; \wedge \; l \in \{ q, q+1, q+2 \} \, \cap \, \{ r, r+1, r+2 \}$.

\bigskip
Situation (A) is much more easier to handle than situation (B), since in situation (A) the first 3-sections of $e$ and $e'$ have a non-empty intersection. Therefore, both hyperedges must be contained in one and the same 4-group $G_{\tj} = \{ e_j, e_{j+1}, e_{j+2}, e_{j+3}  \}$, with $j \in \{1,5,9, \ldots, \frac{n}{2} - 3 \}$ and $\widetilde{j} = \frac{j-1}{4} \in \{ 0,1, \ldots, \frac{n}{8} - 1 \}$. Moreover, $\{e, e' \} \neq \{ e_j, e_{j+3} \}$. Owing to $n \ge 24$, the definition of the hyperedges $e_j, e_{j+1}, e_{j+2}, e_{j+3}$ provides the contradiction $\{ q, q+1, q+2 \} \, \cap \, \{ s, s+1, s+2 \} = \emptyset$.

\medskip
Consequently, situation (B) has to occur and we assume \\ $e= \{ p, p+1, p+2, q, q+1, q+2 \}$, $e' = \{ r, r+1, r+2, s, s+1, s+2 \}$, \\[0.3ex]  $k \in \{ p, p+1, p+2 \} \, \cap \, \{ s, s+1, s+2 \}$ and $l \in \{ q, q+1, q+2 \} \, \cap \, \{ r, r+1, r+2 \}$.

\smallskip
Obviously, $ e \in G_{\tj} = \{ e_j, e_{j+1}, e_{j+2}, e_{j+3} \ \}$ and $ e' \in G_{\widetilde{j'}} = \{ e_{j'}, e_{j'+1}, e_{j'+2}, e_{j'+3} \ \}$  with \\[0.3ex]
$j, j' \in \{1,5,9, \ldots, \frac{n}{2} - 3 \}$ and $\widetilde{j} = \frac{j-1}{4}, \widetilde{j'} = \frac{j'-1}{4} \in \{ 0,1, \ldots, \frac{n}{8} - 1 \}$ as well as  $j \neq j'$, $ \widetilde{j} \neq \widetilde{j'}$.

We discuss all possible choices of $e$ and $e'$ in the sets $\{ e_j, e_{j+1}, e_{j+2}, e_{j+3} \ \}$ and \\ $\{ e_{j'}, e_{j'+1}, e_{j'+2}, e_{j'+3} \ \}$, respectively. In order to find the wanted $k$ and  $l$ in the intersections of the corresponding 3-sections, we are searching for numbers, i.e. vertices, $8 \tj + x$ and $8 \widetilde{j'} + y$ in the 3-sections being under investigation, which have one and the same remainder modulo 8.

\medskip
\noindent
\ul{B1: $e = e_j \; \wedge \; e' = e_{j'}$.}

\smallskip
From  $ k \in \{ p, p+1, p+2 \} \, \cap \, \{ s, s+1, s+2 \}$ \\[0.5ex]
$= \{ 8 \tj + 1, 8 \tj + 2, 8 \tj + 3 \} \, \cap \, \{ 8 \widetilde{ j'} + 7, 8 \widetilde{ j'} + 8, 8 \widetilde{ j'} +  9 \}$  \\[0.5ex]
it follows $ k =  8 \tj + 1 = 8 \widetilde{ j'} +  9$ and, therefore, $\tj = \widetilde{ j'} + 1$.

Thereby, $l \in \{ q, q+1, q+2 \} \, \cap \, \{ r, r+1, r+2 \}$ \\[0.5ex]
$= \{ 8 \tj + 7, 8 \tj + 8, 8 \tj + 9 \} \, \cap \, \{ 8 \widetilde{ j'} + 1, 8 \widetilde{ j'} + 2, 8 \widetilde{ j'} +  3 \}$ \\[0.5ex]
$= \{ 8 \widetilde{ j'} + 15, 8 \widetilde{ j'} + 16, 8 \widetilde{ j'} +  17 \} \, \cap \, \{ 8 \widetilde{ j'} + 1, 8 \widetilde{ j'} + 2, 8 \widetilde{ j'} +  3 \} = \emptyset$, \\
a contradiction.

\medskip
\noindent
\ul{B2: $e = e_j \; \wedge \; e' = e_{j'+1}$.}

\smallskip
From  $ k \in \{ p, p+1, p+2 \} \, \cap \, \{ s, s+1, s+2 \}$ \\[0.5ex]
$= \{ 8 \tj + 1, 8 \tj + 2, 8 \tj + 3 \} \, \cap \, \{ 8 \widetilde{ j'} + 16, 8 \widetilde{ j'} + 17, 8 \widetilde{ j'} +  18 \}$  \\[0.5ex]
it follows $ k =  8 \tj + 1 = 8 \widetilde{ j'} +  17$ or $ k =  8 \tj + 2 = 8 \widetilde{ j'} +  18$ and, therefore, $\tj = \widetilde{ j'} + 2$.

Thereby, $l \in \{ q, q+1, q+2 \} \, \cap \, \{ r, r+1, r+2 \}$ \\[0.5ex]
$= \{ 8 \tj + 7, 8 \tj + 8, 8 \tj + 9 \} \, \cap \, \{ 8 \widetilde{ j'} + 2, 8 \widetilde{ j'} + 3, 8 \widetilde{ j'} +  4 \}$ \\[0.5ex]
$= \{ 8 \widetilde{ j'} + 23, 8 \widetilde{ j'} + 24, 8 \widetilde{ j'} +  25 \} \, \cap \, \{ 8 \widetilde{ j'} + 2, 8 \widetilde{ j'} + 3, 8 \widetilde{ j'} +  4 \}$. \\[0.5ex]
In the case $n = 24$ this intersection is equal to $\{ 8 \widetilde{ j'} - 1, 8 \widetilde{ j'} , 8 \widetilde{ j'} +  1 \} \, \cap \, \{ 8 \widetilde{ j'} + 2, 8 \widetilde{ j'} + 3, 8 \widetilde{ j'} +  4 \} = \emptyset$, and in the case $n \ge 40$ the intersection is trivially equal to $\emptyset$, incompatible to (B).

%\pagebreak

\medskip
\noindent
\ul{B3: $e = e_j \; \wedge \; e' = e_{j'+2}$.}

\smallskip
For all $j, j' \in \{1,5,9, \ldots, \frac{n}{2} - 3 \}$ it follows $\{ p, p+1, p+2 \} \, \cap \, \{ s, s+1, s+2 \}$ \\[0.5ex]
$= \{ 8 \tj + 1, 8 \tj + 2, 8 \tj + 3 \} \, \cap \, \{ 8 \widetilde{ j'} + 13, 8 \widetilde{ j'} + 14, 8 \widetilde{ j'} +  15 \} = \emptyset$.

\medskip
\noindent
\ul{B4: $e = e_j \; \wedge \; e' = e_{j'+3}$.}

\smallskip
For all $j, j' \in \{1,5,9, \ldots, \frac{n}{2} - 3 \}$ it follows $\{ p, p+1, p+2 \} \, \cap \, \{ s, s+1, s+2 \}$ \\[0.5ex]
$= \{ 8 \tj + 1, 8 \tj + 2, 8 \tj + 3 \} \, \cap \, \{ 8 \widetilde{ j'} - 2, 8 \widetilde{ j'} - 1, 8 \widetilde{ j'} \} = \emptyset$.

\medskip
\noindent
\ul{B5: $e = e_{j+1} \; \wedge \; e' = e_{j'+1}$.}

\smallskip
From  $ k \in \{ p, p+1, p+2 \} \, \cap \, \{ s, s+1, s+2 \}$ \\[0.5ex]
$= \{ 8 \tj + 2, 8 \tj + 3, 8 \tj + 4 \} \, \cap \, \{ 8 \widetilde{ j'} + 16, 8 \widetilde{ j'} + 17, 8 \widetilde{ j'} +  18 \}$  \\[0.5ex]
it follows $ k =  8 \tj + 2 = 8 \widetilde{ j'} +  18$ and, therefore, $\tj = \widetilde{ j'} + 2$.

Thereby, $l \in \{ q, q+1, q+2 \} \, \cap \, \{ r, r+1, r+2 \}$ \\[0.5ex]
$= \{ 8 \tj + 16, 8 \tj + 17, 8 \tj + 18 \} \, \cap \, \{ 8 \widetilde{ j'} + 2, 8 \widetilde{ j'} + 3, 8 \widetilde{ j'} +  4 \}$ \\[0.5ex]
$= \{ 8 \widetilde{ j'} + 32, 8 \widetilde{ j'} + 33, 8 \widetilde{ j'} +  34 \} \, \cap \, \{ 8 \widetilde{ j'} + 2, 8 \widetilde{ j'} + 3, 8 \widetilde{ j'} +  4 \} = \emptyset$ -  incompatible to (B), since $n=24$ or $n \ge 40$. \\[0.5ex]
Note that in the case $n=32$ we would have $ l = 8 \widetilde{ j'} +  34 = 8 \widetilde{ j'} +  2$. This is the reason why for $n=32$ a modified construction of the hyperedges of $\cH$ will have  to be used later (cf. Remark 1 at the end of Subsection 2.2.2).

\medskip
\noindent
\ul{B6: $e = e_{j+1} \; \wedge \; e' = e_{j'+2}$.}

\smallskip
For all $j, j' \in \{1,5,9, \ldots, \frac{n}{2} - 3 \}$ it follows  $\{ p, p+1, p+2 \} \, \cap \, \{ s, s+1, s+2 \}$ \\[0.5ex]
$= \{ 8 \tj + 2, 8 \tj + 3, 8 \tj + 4 \} \, \cap \, \{ 8 \widetilde{ j'} + 13, 8 \widetilde{ j'} + 14, 8 \widetilde{ j'} +  15 \} = \emptyset.$

\medskip
\noindent
\ul{B7: $e = e_{j+1} \; \wedge \; e' = e_{j'+3}$.}

\smallskip
For all $j, j' \in \{1,5,9, \ldots, \frac{n}{2} - 3 \}$ we obtain $\{ p, p+1, p+2 \} \, \cap \, \{ s, s+1, s+2 \}$ \\[0.5ex]
$= \{ 8 \tj + 2, 8 \tj + 3, 8 \tj + 4 \} \, \cap \, \{ 8 \widetilde{ j'} - 2, 8 \widetilde{ j'} - 1, 8 \widetilde{ j'} \} = \emptyset.$

\medskip
\noindent
\ul{B8: $e = e_{j+2} \; \wedge \; e' = e_{j'+2}$.}

\smallskip
From  $ k \in \{ p, p+1, p+2 \} \, \cap \, \{ s, s+1, s+2 \}$ \\[0.5ex]
$= \{ 8 \tj + 3, 8 \tj + 4, 8 \tj + 5 \} \, \cap \, \{ 8 \widetilde{ j'} + 13, 8 \widetilde{ j'} + 14, 8 \widetilde{ j'} +  15 \}$  \\[0.5ex]
we get $ k =  8 \tj + 5 = 8 \widetilde{ j'} +  13$ and, therefore, $\tj = \widetilde{ j'} + 1$.

Thereby, $l \in \{ q, q+1, q+2 \} \, \cap \, \{ r, r+1, r+2 \}$ \\[0.5ex]
$= \{ 8 \tj + 13, 8 \tj + 14, 8 \tj + 15 \} \, \cap \, \{ 8 \widetilde{ j'} + 3, 8 \widetilde{ j'} + 4, 8 \widetilde{ j'} +  5 \}$ \\[0.5ex]
$= \{ 8 \widetilde{ j'} + 21, 8 \widetilde{ j'} + 22, 8 \widetilde{ j'} +  23 \} \, \cap \, \{ 8 \widetilde{ j'} + 2, 8 \widetilde{ j'} + 3, 8 \widetilde{ j'} +  4 \} = \emptyset$ --  contradictory to (B).

\medskip
\noindent
\ul{B9: $e = e_{j+2} \; \wedge \; e' = e_{j'+3}$.}

\smallskip
For all $j, j' \in \{1,5,9, \ldots, \frac{n}{2} - 3 \}$ we have $\{ p, p+1, p+2 \} \, \cap \, \{ s, s+1, s+2 \}$ \\[0.5ex]
$= \{ 8 \tj + 3, 8 \tj + 4, 8 \tj + 5 \} \, \cap \, \{ 8 \widetilde{ j'} - 2, 8 \widetilde{ j'} - 1, 8 \widetilde{ j'} \} = \emptyset.$

\medskip
\noindent
\ul{B10: $e = e_{j+3} \; \wedge \; e' = e_{j'+3}$.}

\smallskip
From  $ k \in \{ p, p+1, p+2 \} \, \cap \, \{ s, s+1, s+2 \}$ \\[0.5ex]
$= \{ 8 \tj + 4, 8 \tj + 5, 8 \tj + 6 \} \, \cap \, \{ 8 \widetilde{ j'} - 2, 8 \widetilde{ j'} - 1, 8 \widetilde{ j'} \}$  \\[0.5ex]
we get $ k =  8 \tj + 6 = 8 \widetilde{ j'} - 2$ and, therefore, $\widetilde{ j'} = \tj + 1$.

Thereby, $l \in \{ q, q+1, q+2 \} \, \cap \, \{ r, r+1, r+2 \}$ \\[0.5ex]
$= \{ 8 \tj - 2, 8 \tj - 1, 8 \tj \} \, \cap \, \{ 8 \widetilde{ j'} + 4, 8 \widetilde{ j'} + 5, 8 \widetilde{ j'} +  6 \}$ \\[0.5ex]
$= \{8 \tj - 2, 8 \tj - 1, 8 \tj  \} \, \cap \, \{ 8 \widetilde{ j} + 12, 8 \widetilde{ j} + 13, 8 \widetilde{ j} +  14 \} = \emptyset$, in contradiction to (B).

\medskip
Thus $E(C_n) = \cE(EI(\cH))$ and the proof of 2.2.1.1 (basic construction) is complete.

\medskip
We remark that the 6-uniformity of $\cH$ is trivial (see 2.2.1, (I)--(IV), for the definition of the hyperedges
$e_j$, $j \in \{ 1,2, \ldots, \frac{n}{2} \}$).

A second remark concerns the 3-regularity of   $\cH$, which is a conclusion from
\begin{itemize}
\item
$|V| = n$ and $| \cE | = \frac{n}{2}$,
\item
the numbering of the vertices along the cycle $C_n$ by $1,2, \ldots,n$,
\item
every hyperedge consists of two 3-sections of vertices, where each 3-section contains three immediately consecutive vertices,
\item
for each pair  of distinct hyperedges $e_f, e_g \in \cE$ it holds that $e_f$ and $e_g$ do never have a 3-section in common.
\end{itemize}
These properties imply that the $n$ 3-sections of all $\frac{n}{2}$ hyperedges of $\cE$ have to form the set $\{ \{1,2,3\}, \{ 2,3,4 \}, \ldots, \{n-2, n-1, n \}, \{n-1, n, 1 \}, \{n, 1, 2 \} \}$. Therefore each vertex $ v \in V$ is contained in exactly three hyperedges.
\\[1ex]

\paragraph{2.2.1.2 Supplemental construction: $l \in \{2,4,6 \}  $.}
\label{Sec_L12}
$\/$

\bigskip
First of all, we sketch the idea of the proof.

Consider the definitions of the hyperedges
of the hypergraph $\cH$ with $n=8k+l$ vertices given in 2.2.1, (I)--(IV). To avoid the
lengthy verification of  $EI(\cH) = C_n = C_{8k+l}$ analogously to the above proof  of the basic construction (see 2.2.1.1),
our present proof will be done inductively by adding two vertices and one hyperedge to $\cH$ in each induction step.

As the initial induction step, the basic construction ($l=0$) includes the verification for $n = 8 k + l = 8 k$.
For the induction hypothesis, let $l \in \{2, 4, 6 \}$ and  $\cH$ be the hypergraph constructed in 2.2.1, (I)--(IV),  with $n-2 = 8k + l - 2$
vertices, which has the property $EI(\cH) = C_{n-2} = C_{8k + l - 2}$.
We show how
to construct a hypergraph $\cH'$ with $n=8k + l$ vertices and $EI(\cH') = C_n$  from the  hypergraph $\cH$.

 For this end, we will add two new
 vertices $x$ and $y$ as well as one new hyperedge $e_\ast$ to $\cH$ and obtain $\cH'= (V', \cE') = (V(\cH)
  \cup \{x, y \}, \, \cE(\cH) \cup \{ e_\ast \} )$.
  This provides the first (the constructive) part of our proof (Step 1).
%
%\bigskip
% \noindent
%{\bf\ul{Step 1.}} Construction of the set of hyperedges $\cE(\cH)$ of the hypergraph $\cH$.

The reason is that  -- after relabelling the vertices and the hyperedges -- it will be easy to see that  the hyperedges of $\cH'$ are exactly the hyperedges given in our construction in 2.2.1, (I)--(IV).
Therefore, Step 1 from the beginning of 2.2.1 can be used also in the present case, i.e. for the supplemental construction.

\bigskip
Then, we will have to verify  $EI(\cH') = C_{n}$ for each $l \in \{2,3,4 \}$ (this corresponds to Step 2 and Step 3). For technical reasons, in the Subcase $l=2$
we will treat Step 3 before Step 2. The 6-regularity and 3-uniformity of the hypergraph $\cH$ follows analogously to the basic construction, i.e. to 2.2.1.1.

\medskip
As a preliminary consideration, we prove the following.

\medskip
\noindent
{\bf Claim.}
{\em Let $n= 8 k$, and $j, \tj, G_{\widetilde{j}} = \{ e_j, e_{j+1}, e_{j+2}, e_{j+3} \}$ be defined as at the beginning of the basic construction.
Then there is no hyperedge $e_z \in \cE$ with $ \{ 8 \tj + 1, 8 \tj + 2 \} \, \cap \,e_z \neq \emptyset$ and \\[0.3ex] $ \{ 8 \tj + 5, 8 \tj + 6 \} \, \cap \, e_z \neq \emptyset .$}

\begin{proof}
The definitions of $e_j, e_{j+1}, e_{j+2}$ and  $e_{j+3}$ provide $e_z \in \cE \setminus G_{\widetilde{j}}$.

Assume $p \in  \{ 8 \tj + 1, 8 \tj + 2 \} \, \cap \,e_z $ and $ q \in \{ 8 \tj + 5, 8 \tj + 6 \} \, \cap \, e_z$.
%Then $3 \le q - p \le 5$.
Then there is a $j' \in \{ 1,5,9, \ldots, \frac{n}{2} -3 \} \setminus \{ j \}$ such that $z \in \{ j', j'+1, j'+2, j'+3 \}$ and $\widetilde{j'} = \frac{j'-1}{4} \neq \tj.$

Obviously,
the remainder of $p$ and $q$ modulo 8 is from the set $\{ 1, 2 \}$  and $\{ 5, 6 \}$, respectively.

\noindent
Since the sets of the remainders modulo 8 of the vertices in the two 3-sections of $e_z$ are

\smallskip
$ \{ 1, 2, 3 \}$ and $ \{ 7, 0, 1 \}$, if $z = j'$,

$ \{  2, 3, 4 \}$ and $ \{  0, 1, 2 \}$, if $z = j'+1$,

$ \{  3, 4, 5 \}$ and $ \{ 5, 6, 7 \}$, if $z = j'+2$ \; and

$ \{ 4, 5, 6 \}$ and $ \{ 6, 7, 0 \}$, if $z = j'+3$, \\[0.5ex]
$p, q \in e_z$ is impossible. \hqed
\end{proof}

\smallskip

For simplicity, first we investigate the subcase $n = 8 k +2$, i.e. $l=2$. The remaining cases will make use of an analogous construction.

\bigskip
\noindent
{\ul {\em Subcase  $l = 2 $.}}

\medskip
\medskip
Again,  Step 1 (see (I)--(IV) at the beginning of 2.2.1) has to be adapted to the present case $l=2$.

\pagebreak
 \noindent
{\bf\ul{Step 1.}} Construction of the set of hyperedges $\cE(\cH)$ of the hypergraph $\cH$.

\medskip
Note that also for $l=2$ we will give an example (Example 2, see at the end of the subcase) for the constructed hypergraph, namely for $n=26$. But first, we will describe the idea of our construction.

For this end, let $n = 8k + 2$, $\cH =(V, \cE)$ with $V= \{ 1,2, \ldots, n-2 \}$ and $\cE = \{ e_1, e_2, \ldots, e_{\frac{n-2}{2}} \}$, where the hyperedges in $\cE$ should be constructed as in 2.2.1, (I)--(IV) (now with $n-2$ instead of $n$).
Consequently, $EI(\cH) = C_{n-2}$.

\medskip
To give a sketchy idea, consider the  4-group $G_{0} = \{ e_1, e_2, e_3, e_4 \}$ consisting of the hyperedges \\[0.5ex]
$e_1 = \{ 1,2,3,7,8,9 \}$,\\[0.3ex]
$e_2 = \{ 2,3,4,16,17,18 \}$,\\[0.3ex]
$e_3 = \{ 3,4,5, 13,14,15 \}$,\\[0.3ex]
$e_4 = \{ 4,5,6, n-4, n-3, n-2 \} = \{ 4,5,6, 8k - 2, 8k - 1, 8k \}$.

\smallskip
We proceed as follows: we insert the new vertex $x$  "between" the (old) vertices $6$ and $7$ (remember that the vertices are numbered by $1,2, \ldots, n-2$ along the cycle $C_{n-2}$). Then we insert a second new vertex $y$ "between" the (old) vertices $n-2$ and $1$. Simultaneously, we add a new hyperedge $e_\ast $ to the 4-group $G_0$ such that the new $G_0$ becomes a 5-group, i.e. $G_{0} = \{ e_1, e_2, e_3, e_4, e_\ast \}$. This procedure results in a new hypergraph $\cH'$ with -- as we will prove -- $EI(\cH') = C_{n}$.

To this end, a few modifications of several (old) hyperedges of the original hypergraph $\cH$ and some relabelling of vertices and hyperedges have to be carried out.

\bigskip
Now we come to the details of our construction.

\begin{enumerate}
\item[(i)]
We preliminarily remark that, owing to the above Claim, adding a new hyperedge $ e' = \{1,2,5,6 \}$ to $\cH$ would not induce a chord in $EI(\cH \cup \{ e' \}) = EI(\cH) = C_{n-2}$.

\smallskip
\item[(ii)]
We relabel the vertices of $V$ (also inside the hyperedges of $\cH$) according to

\smallskip
 $v := \left\{ \begin{array}{l@{\quad,\quad}l}
 v & v \in \{1,2, \ldots, 6 \} \\[0.8ex] v + 1  & v \in \{ 7,8, \ldots, n-2 \}
 \end{array} \right.$ \\

Then we have $ V = \{ 1,2, \ldots, 6,8,9, \ldots, n-1 \}$ and (i) remains valid (with this modified numbering of the vertices).

\smallskip
\item[(iii)]
Now we add two vertices and one hyperedge to the hypergraph $\cH =( V, \cE)$ and obtain $\cH'=(V', \cE')$ with

\smallskip
$V' \; := \; V \, \cup \{ x, y \}$, where $x = 7$ and $ y = n$ \quad and

%\smallskip
%$e_\ast \; := \; e' \, \cup \, \{ 7, n \} \; = \, \{ 5,6,7, n,1,2 \}$ and

%\smallskip
$\cE' \; := \; \cE \, \cup \, \{ e_\ast \}$, where $e_\ast \; := \; e' \, \cup \, \{ 7, n \} \; = \, \{ 5,6,7, n,1,2 \}$.

\smallskip
Following our usual notation, we refer to $\{ 5,6,7 \}$ and to $\{n,1,2 \}$ as the first and the second 3-section of $e_\ast$, respectively.

Because of (i) and the fact that $e_\ast$ is the only hyperedge containing the new vertices $7$ and $n$, (i) remains valid in the sense that there are no vertices $k, l \in V'$ with  $|k - l| > 1$ and $\{ k, l \} \subset e$, for all hyperedges $e \in \cE(EI(\cH'))$. The only exception is the (unoffending) case $k = 1$ and $l=n$. In other words: the new hyperedge $e_\ast$ will not generate a chord in the cycle $C_n$, which will arise as $EI(\cH')$ in (iv) (see below).

\medskip
In Step 3(i)  of the basic construction we saw that no two hyperedges have a 3-section in common. Obviously, this remains valid up to now, also if we include the new hyperedge $e_\ast$. Therefore, each hyperedge $e \in \cE(\cH') $ is uniquely determined by giving its first or its second 3-section. This property will simplify the considerations in  (iv).

\medskip
Finally, note that the relabelling of the vertices in (ii) and the insertion of the new vertices $7$ and $n$ in (iii) induce that in some 3-sections of several hyperedges there is a gap between the vertices $6, 8$ and $n-1, 1$, respectively. Originally, each 3-section of a hyperedge consists of three vertices $p, p+1, p+2$ being consecutive on the cycle $C_n$. The next step of our construction retrieves this previous state for all hyperedges.

\smallskip
\item[(iv)]
 Some 3-sections of four special hyperedges $e_\alpha, e_\beta, e_\gamma, e_\delta$ have to be modified and the modifications exclusively bear on the second 3-sections of these hyperedges. For clearness, in each case we give the modifications in the form

$ e_\kappa \; : \; \{ x, y, z \} \; \Rightarrow \; \{ x', y', z' \} \; \Rightarrow \; \{ x'', y'', z'' \}$, where $\kappa \in \{ \alpha, \beta, \gamma, \delta \}$ and

\smallskip
$ \{ x, y, z \}$ is the original 3-section in the hypergraph $\cH$ (before (i)),

$\{ x', y', z' \}$ is the 3-section after  (ii) (this is the same as after  (iii)) and

$\{ x'', y'', z'' \}$ is the final form of the 3-section after  (iv).

\smallskip
As mentioned above, the other (these are the first) 3-sections of the modified hyperedges remain unchanged. They  are not needed for our argumentation, so in the following we present only the (modified) second 3-sections.

\medskip
$ e_\alpha \; : \; \{ 5, 6, 7 \}  \hspace{15.6mm}  \Rightarrow \; \{ 5, 6, 8 \}   \hspace{15.6mm}  \Rightarrow \; \{ 6, 7, 8 \}$,

\smallskip
$ e_\beta \; : \; \{ 6, 7, 8 \}  \hspace{15.6mm} \Rightarrow \; \{ 6, 8, 9 \}  \hspace{15.6mm} \Rightarrow \; \{ 7, 8, 9 \}$,

\smallskip
$ e_\gamma \; : \; \{ n-3, n-2, 1 \} \; \, \Rightarrow \; \, \{ n-2, n-1, 1 \} \; \Rightarrow \; \{ n-2, n-1, n \}$ \quad and

\smallskip
$ e_\delta \; : \; \{  n-2, 1,2 \}  \hspace{9.1mm} \Rightarrow \; \{ n-1, 1,2 \}  \hspace{8.9mm} \Rightarrow \; \{n-1, n,1 \}$.\\[3ex]
{\bf\ul{Step 3.}}  The hyperedges of $\cH'$ do not generate any chord in $C_n$: $\cE( EI(\cH')) \subseteq E(C_n)$.

\smallskip
Note that $e_{\alpha} \neq e_\ast \,$, since $e_\ast$ comes into play not before (iii).
Clearly, considering only the hyperedges in $\cE \setminus \{ e_\alpha, e_\beta, e_\gamma, e_\delta, e_\ast \}$, these hyperedges  induce a subgraph of $C_n$ in $EI( \cH')$, therefore they do not cause any chords. Owing to (iii), no chord emerges if we add $e_\ast$ to this set of hyperedges.

Hence, we only have to show that the new hyperedges $e_\alpha, e_\beta, e_\gamma, e_\delta$ do not generate a chord:

\begin{enumerate}
\item
Obviously, the relabelling of the vertices in (ii) does not result in a chord, since the relabelling is carried out simultaneously in all hyperedges, i.e. -- apart from the modified vertex numbers -- the hyperedges remain unchanged.
\item
It is clear that the addition of the new vertices $7$ and $n$ in  (iii) cannot lead to a chord, because none of the "old" hyperedges (i.e. the hyperedges in $\cE$) contains one of these vertices.
\item
Above we demonstrated that, in (iii), $e_\ast$ does not generate any chord.
\item
It suffices to show that the modifications in (iv) do not produce any chord. Since these modifications involve only the hyperedges $e_\alpha, e_\beta, e_\gamma, e_\delta$, we have to consider solely the intersections of these hyperedges with other hyperedges of $\cH'$.
\begin{enumerate}
\item[(d1)]
It can be easily seen that (analogously as in 2.2.1.1, Step 3(i)), for all hyperedges $e_f \neq e_g$ in $\cE' = \cE \cup \{ e_\ast \}$, the hyperedges $e_f$ and $e_g$ do never have a 3-section in common. Therefore, a chord $\{k, l \}$ could result only from the intersection of two hyperedges $e_f$ and $e_g$, where $k$ is contained in one 3-section of $e_g$ as well as in one 3-section of $e_f$ and $l$ is contained in the other 3-section of $e_g$ and the other 3-section of $e_f$.

\smallskip
\item[(d2)]
We consider $e_\alpha = \{ \ldots, 6,7,8 \}$.

In (iv), the replacement of the vertex 5 by the vertex 7 does not result in a chord, since at that moment 7 was only in the hyperedge $e_\ast$ and, moreover, already before (iv) we had a nonempty intersection $e_\ast \cap e_\alpha \supseteq \{ 5, 6 \}$. Therefore -- owing to the non-existence of chords at this time -- the other 3-sections of $e_\ast$ and $e_\alpha$ have to be disjoint.

\smallskip
\item[(d3)]
Lets  look at $e_\beta = \{ \ldots, 7,8,9 \}$.

In (iv), superseding 6 by 7 does not lead to a chord, because till then 7 had been only in the hyperedges $e_\ast$ and $e_\alpha$. Additionally, before (iv) we had also nonempty intersections  $e_\ast \cap e_\beta \supseteq \{ 6 \}$ and $e_\alpha \cap e_\beta \supseteq \{ 6,8 \}$. Hence, for the same reason as in (d2), the other 3-sections of  $e_\ast$ and  $e_\beta$ as well as of $e_\alpha$ and $e_\beta$ must be disjoint.

\smallskip
\item[(d4)]
We come to $e_\gamma = \{ \ldots, n-2, n-1, n \}$.

The substitution of the vertex 1 by $n$ in (iv) does not generate a chord. The reason is that before this substitution the vertex $n$ had been contained only in $e_\ast$. What is more, before (iv) obviously $e_\ast \cap e_\gamma \supseteq \{ 1 \} \neq \emptyset$ and, analogously to (d1) and (d2), the remaining 3-sections of both hyperedges cannot include any common vertex.

\smallskip
\item[(d5)]
Finally, we investigate $e_\delta = \{ \ldots, n-1, n, 1 \}$.

Now we look at the replacement of the vertex 2 by $n$ in (iv) and see that until then $n$ had been only an element of the hyperedges $e_\ast$ and $e_\gamma$. Again, before (iv) we had  nonempty intersections  $e_\ast \cap e_\delta \supseteq \{ 1,2 \}$ and  $e_\gamma \cap e_\delta \supseteq \{ n-1, 1 \}$, consequently the intersections of the other 3-sections of these hyperedges have to be empty.

\end{enumerate}

\smallskip
So we see that  (ii), (iii) and (iv) in our construction do not generate any chord in $EI(\cH')$.

\end{enumerate}

\smallskip
\noindent
{\bf\ul{Step 2.}} {\em   The hyperedges of $\cH'$ generate all edges of $C_n$: $E(C_n) \subseteq \cE( EI(\cH'))$.}

\smallskip
In order to verify that  $E(C_n) \subseteq \cE( EI(\cH'))$, we do not want to discuss here the -- in a certain sense "trivial" -- edges $\{i, i+1 \}$ of $C_n$ which result from "simply incremented vertices" (see (ii)). Before the relabelling step (ii), such edges had been edges of the cycle $C_{n-2} = EI(\cH)$. Therefore, it suffices to consider the remaining edges of ´$C_n$, i.e. the edges in the vertex ranges $\{ 5,6,7,8,9 \}$ and $\{n-2, n-1, n, 1, 2 \}$.

These edges result from the following intersections:

$ \{  5 ,6   \} = e_4  \, \cap \, e_\ast     $, $ \{  6 , 7  \} = e_\ast  \, \cap \, e_\alpha     $, $ \{ 7  , 8  \} = e_\alpha   \, \cap \, e_\beta     $, $ \{  8 , 9  \} = e_\beta   \, \cap \,  e_1     $, \\ $ \{ n-2  , n-1  \} = e_4   \, \cap \,  e_\gamma     $, $ \{ n-1  , n  \} = e_\gamma    \, \cap \, e_\delta     $, $ \{ n  , 1   \} = e_\delta   \, \cap \,  e_\ast     $ and $ \{ 1   , 2  \} = e_\ast    \, \cap \,  e_1     $.

\end{enumerate}

Now we come to Example 2. Note that  for the hyperedges  the labelling from
 2.2.1, (I)--(IV), is used, i.e. the hyperedge $e_\ast$ becomes $e_5$ and the indices of the former $e_5, e_6, \ldots$ have to be increased by 1.

\pagebreak
\begin{example}
$\cH = (V, \{e_1, e_2, \ldots, e_{13} \})$ with $EI(\cH) = C_{26}$ has the following hyperedges.\\[1ex]
$ \begin{array}{l@{\, = \, \{ \,}c@{, \,}c@{, \,}c@{, \,}c@{, \,}c@{, \,}c@{\, \},}}
e_{1} & {\bf 1} & {\bf 2} & {\bf 3} &  8  & 9  & 10 \\[0.5ex]
e_{2} & {\bf 2} & {\bf 3} & {\bf 4} & 17 & 18 & 19 \\[0.5ex]
e_{3} & {\bf 3} & {\bf 4} & {\bf 5} &  14 & 15 & 16 \\[0.5ex]
e_{4} & {\bf 4} & {\bf 5} & {\bf 6} &  23  & 24  & 25  \\[0.5ex]
e_{5} &  1  & 2 & {\bf 5} & {\bf 6} & {\bf 7} &  26   \\[0.5ex]
e_{6} & {\bf 10} & {\bf 11} & {\bf 12} &  16  & 17  & 18 \\[0.5ex]
e_{7} & 1 & {\bf 11} & {\bf 12} & {\bf 13} & 25 & 26 \\[0.5ex]
e_{8} & {\bf 12} & {\bf 13} & {\bf 14} &  22 & 23 & 24 \\[0.5ex]
e_{9} & 7 & 8 & 9 &  {\bf 13}  & {\bf 14}  & {\bf 15}  \\[0.5ex]
e_{10} & {\bf 18} & {\bf 19} & {\bf 20} &  24  & 25  & 26 \\[0.5ex]
e_{11} & 9 & 10 & 11 & {\bf 19} & {\bf 20} & {\bf 21} \\[0.5ex]
e_{12} & 6 & 7 & 8 &  {\bf 20} & {\bf 21} & {\bf 22} \\[0.5ex]
e_{13} & 15 & 16 & 17 &  {\bf 21}  & {\bf 22}  & {\bf 23}
\end{array} $ \\[1.5ex]
where in $EI(\cH)$ the edges of $C_{26}$ are generated analogously as shown in Example 1.

Note that now $G_0 = \{e_1, e_2, e_3, e_4, e_5 \}$ is a 5-group and the other groups $G_1$ and  $G_2$  are 4-groups of hyperedges.

\end{example}

\medskip
\noindent
{\ul {\em Subcase $l =  4  $.}}

\medskip
Here is the adaption of  Step 1  to the actual case $l=4$.

\medskip
 \noindent
{\bf\ul{Step 1.}} Construction of the set of hyperedges $\cE(\cH)$ of the hypergraph $\cH$.

\medskip
At the outset of 2.2.1.2, we remarked that Step 1 from the beginning of 2.2.1 can be \\[0.3ex] used also in 2.2.1.2, i.e. for each $l \in \{ 2, 3, 4 \}$ in the supplemental construction.

\smallskip
We proceed analogously to Subcase $l=2$, so we make use of our construction of the hypergraph $\cH =( V, \cE)$ with $n-2$ vertices, where now $ n-2 = (8k + 4 ) -2 = 8 k +2$ holds. Looking at Example 2, we remember that $G_0 = \{e_1, e_2, e_3, e_4, e_5 \}$ is a 5-group and the other groups $G_1, G_2, \ldots$ are 4-groups.

In a first step, we relabel \vspace{-1mm}
\begin{itemize}
\item the vertices $ v \in V = \{ 1,2, \ldots, n-2 \}$ by $v := v - 9$ modulo $(n-2)$,
\item the hyperedges $e_i \in \cH = \{e_1, e_2, \ldots, e_{\frac{n-2}{2}} \}$ by
$e_i := e_{i-4}$ (the indices taken modulo $\frac{n-2}{2}$) and
\item the edge groups $G_j \in \{ G_0, G_1, \ldots G_{\frac{n-4}{8} - 1} \}$ by
$G_j := G_{j-1}$ (the indices taken modulo $\frac{n-4}{8}$).
\end{itemize}

That way, the former 4-group $G_{\frac{n-4}{8} - 1}$ becomes the 4-group $G_0$ with the new label $0$ and the former 5-group $G_0$ is now the new  5-group $G_1$ (with the corresponding relabeled vertices and hyperedges). Consequently, in our present hypergraph we have locally (i.e., in the vertex range, where in Subcase $l=2$ the essential modifications in $\cH$ in connection with the insertion of the new vertices $x$ and $y$
and the new hyperedge $e_\ast$ took place) the same structure  as we had in Subcase $l=2$.
Therefore, also in the present situation (i.e., where $n = 8k +4$ holds) the same procedure as in Subcase $l=2$  can be used in order to obtain a new hypergraph $\cH'$ with $n$ vertices from the hypergraph $\cH$ with $n-2$ vertices.

\bigskip
\noindent
\ul{\bf{Step 2 / Step 3.}}
Analogously to Subcase $l=2$, it can be shown  that $EI(\cH') = C_{n}$ is valid. We omit the detailed proof here.

\pagebreak
Choosing $n=28$ we give the last example in 2.2.1.2. Again
$e_\ast$ becomes $e_5$ and the indices of the former $e_5, e_6, \ldots$ have to be increased by 1.

\medskip
\begin{example}
$\cH = (V, \{e_1, e_2, \ldots, e_{14} \})$ with $EI(\cH) = C_{28}$ has the following hyperedges.\\[1ex]
$ \begin{array}{l@{\, = \, \{ \,}c@{, \,}c@{, \,}c@{, \,}c@{, \,}c@{, \,}c@{\, \},}}
e_{1} & {\bf 1} & {\bf 2} & {\bf 3} &  8  & 9  & 10 \\[0.5ex]
e_{2} & {\bf 2} & {\bf 3} & {\bf 4} & 19 & 20 & 21 \\[0.5ex]
e_{3} & {\bf 3} & {\bf 4} & {\bf 5} &  16 & 17 & 18 \\[0.5ex]
e_{4} & {\bf 4} & {\bf 5} & {\bf 6} &  25  & 26  & 27  \\[0.5ex]
e_{5} &  1  & 2 & {\bf 5} & {\bf 6} & {\bf 7} &  28   \\[0.5ex]
e_{6} & {\bf 11} & {\bf 12} & {\bf 13} &  18  & 19  & 20 \\[0.5ex]
e_{7} & 1 & {\bf 12} & {\bf 13} & {\bf 14} & 27 & 28 \\[0.5ex]
e_{8} & {\bf 13} & {\bf 14} & {\bf 15} &  24 & 25 & 26 \\[0.5ex]
e_{9} & 7 & 8 & 9 &  {\bf 14}  & {\bf 15}  & {\bf 16}  \\[0.5ex]
e_{10} & 10  & 11  & 12 & {\bf 15} & {\bf 16} & {\bf 17}\\[0.5ex]
e_{11} & {\bf 20} & {\bf 21} & {\bf 22} & 26 & 27 & 28   \\[0.5ex]
e_{12} & 9 & 10 & 11 &  {\bf 21} & {\bf 22} & {\bf 23} \\[0.5ex]
e_{13} & 6 & 7 & 8 &  {\bf 22}  & {\bf 23}  & {\bf 24}\\[0.5ex]
e_{14} & 17 & 18 & 19 &  {\bf 23}  & {\bf 24}  & {\bf 25}
\end{array} $ \\[1.5ex]
where now $G_0 = \{e_1, e_2, e_3, e_4, e_5 \}$ and $G_1 = \{e_6, e_7, e_8, e_9, e_{10} \}$ are 5-groups and \\ $G_2 = \{e_{11}, e_{12}, e_{13}, e_{14} \}$  is a 4-group of hyperedges.

\end{example}

\medskip
\noindent
{\ul {\em Subcase $l =  6  $.}}

\medskip
The verification can be done  analogously to Subcase $l=4$ . \hqed
%\\[1.5ex]
%Concerning the 6-uniformity and the 3-regularity the argumentation given at the end of the basic construction (see 3.2.1.1) remains valid also for the supplemental construction in 3.2.1.2.

\medskip

\subsubsection{Proof of Lemma 2}
\label{Sec_L2}
$\/$

\bigskip
Now we have $k=4$ and therefore $n = 8 k + l = 32 + l$ is valid.
%At first we want to explain why
Remember that the construction used in the proof of Lemma \ref{Leven} (cf. 2.2.1) cannot be applied in the present case (see the note in 2.2.1.1, Step 3, B5).

%To see the problem, it suffices to consider $n = 32$ and have a look at  the construction %in the proof  of Lemma \ref{Leven} (cf. the above case (C) for $l=0$).
%in 3.2.1, part (III).
%Then in one case the construction of $e_{j+1}$ and $e_{j'+1}$ for  certain $j \neq j'$ causes the conflict that $e_{j+1} \, \cap \, e_{j'+1}$ will be a chord in $C_n$. The reason is that if $ \tj = \tj' + 2$ holds, then because of \\[0.5ex]
%$\{ 8 \tj +2, 8 \tj +3, 8 \tj +4 \} \, \cap \, \{ 8 \tj' + 16,  8 \tj' + 17,  8 \tj' + 18 \} = \{ 8 \tj +2 \}$ \quad and \\[0.5ex]
%$\{ 8 \tj' +2, 8 \tj' +3, 8 \tj' +4 \} \, \cap \, \{ 8 \tj + 16,  8 \tj + 17,  8 \tj + 18 \} = $ \\[0.5ex]
%$\{ 8 \tj' +2, 8 \tj' +3, 8 \tj' +4 \} \, \cap \, \{ 8 \tj' + 32,  8 \tj' + 33,  8 \tj' + 34 \} = \{ 8 \tj' +2 \}$ (modulo 32), \\[0.5ex]
% we would obtain the chord
%$\{ 8 \tj +2, 8 \tj' +2  \} = \{ 8 \tj' +18, 8 \tj' +2  \} $.

Hence we have to modify slightly the construction of the hyperedges.

\bigskip
\noindent
{\bf\ul{Step 1.}} Construction of the set of hyperedges $\cE(\cH)$ of the hypergraph $\cH$.

\medskip
Take the construction of the hyperedges  described in Step 1 of the proof of Lemma \ref{Leven} (see 2.2.1). Let $G_{\tj} = \{ e_j, e_{j+1}, e_{j+2}, \ldots \}$ be an arbitrarily chosen 4-group or 5-group of hyperedges.
Then we swap the second 3-sections of the second hyperedge $e_{j+1}$  and the third hyperedge $e_{j+2}$. In detail, for each case we give the modified hyperedges.
Subject to $n > 30$, we use the same distinction of the cases (I)--(IV) as above in 2.2.1. Again, we begin with the 5-groups $G_{\tj}$.

%\pagebreak
\medskip
\noindent
\underline{(I): $\tj \in \{ 0,1, \ldots, \frac{l}{2} - 2 \}.$}

\smallskip
\noindent
$e_{j+1}= \{ 10 \tj+2, 10 \tj+3, 10 \tj+4, 10 \tj+16, 10 \tj+17, 10 \tj+18 \}, $\\[0.5ex]
$e_{j+2}= \{ 10 \tj+3, 10 \tj+4, 10 \tj+5, 10 \tj+19, 10 \tj+20, 10 \tj+21 \}. $

\bigskip

\noindent
\underline{(II): $l \ge 2 \, \wedge \, \tj = \frac{l}{2} - 1.$}

\smallskip
\noindent
 $e_{j+1}= \{ 10 \tj+2, 10 \tj+3, 10 \tj+4,10 \tj+14, 10 \tj+15, 10 \tj+16 \}, $\\[0.5ex]
$e_{j+2}= \{ 10 \tj+3, 10 \tj+4, 10 \tj+5, 10 \tj+17, 10 \tj+18, 10 \tj+19 \}.$

%\pagebreak
Now we come to the 4-groups.

\bigskip

\noindent
\underline{(III): $\tj \in \{ \frac{l}{2}, \frac{l}{2} + 1 \ldots, 2\}$ \; or \; $l=0 \, \wedge \,  \tj = 3.$}

\smallskip
\noindent
$e_{j+1}= \{ x+2, x+3, x+4, x+13, x+14, x+15 \}, $\\[0.5ex]
$e_{j+2}= \{ x+3, x+4, x+5, x+16, x+17, x+18 \}. $

\bigskip

\noindent
\underline{(IV): $l \ge 2 \, \wedge \,  \tj = 3.$}

\medskip
\noindent
$e_{\frac{n}{2} - 2}= \{ n-7,n-6, n-5, 6, 7, 8  \}, $\\[0.5ex]
$e_{\frac{n}{2} - 1}= \{ n-6,n-5, n-4, 9,10,11 \}. $

\medskip
Thereby, Step 1 is complete.
%\end{proof}

\bigskip
\noindent
It remains to demonstrate\\[2ex]
{\bf\ul{Step 2 / Step 3.}} {\em   The hyperedges of $\cH$ generate exactly the edges of $C_n$: $E(C_n) = \cE( EI(\cH))$.}

\medskip
Looking at the construction of the hyperedges above, it is obvious that there are only little modifications in comparison with the construction in the proof of Lemma 1 (cf. 2.2.1), i.e. with the case $n = 8 k + l$, where $k \neq 4$. In principle, for any 4-group or 5-group $G_{\tj} = \{ e_j, e_{j+1}, e_{j+2}, \ldots \}$ of hyperedges (for $n = 8 k + l$, $k \neq 4$) we have only to swap the second 3-sections of the second hyperedge $e_{j+1}$  and the third hyperedge $e_{j+2}$ in order to obtain the corresponding (new) hyperedges $e_{j+1}$ and $e_{j+2}$ for  $n = 8 k + l$, $k= 4$.
%For details, see the (I)--(IV) in the proof of Lemma 2 above.

Hence we could argue that the verification of Steps 2 and 3 for Lemma 2 can be done analogously to that of Lemma 1. Of course, that way the detailed verification for the present case would be as lengthy as in the case of Lemma 1.

But since Lemma 2 includes only the four possible values 32, 34, 36, 38  for the number $n$ of the vertices, alternatively the hyperedges can be written down easily and their intersections can be computed. We give explicitly the edge set of the wanted hypergraph $\cH$ with $EI(\cH) = C_{n}$ for the first two values of $n$, i.e. for $n= 32$ and $n=34$.

%\pagebreak
\begin{example}

$\cH = (V, \{e_1, e_2, \ldots, e_{16} \})$ with $EI(\cH) = C_{32}$ has the following hyperedges.\\[2ex]
$ \begin{array}{l@{\, = \, \{ \,}c@{, \,}c@{, \,}c@{, \,}c@{, \,}c@{, \,}c@{\, \},}}
e_{1} & {\bf 1} & {\bf 2} & {\bf 3} &  7  & 8  & 9 \\[0.5ex]
e_{2} & {\bf 2} & {\bf 3} & {\bf 4} &  13 & 14 & 15 \\[0.5ex]
e_{3} & {\bf 3} & {\bf 4} & {\bf 5} &  16 & 17 & 18\\[0.5ex]
e_{4} & {\bf 4} & {\bf 5} & {\bf 6} &  30  & 31  & 32  \\[0.5ex]
e_{5} & {\bf 9} & {\bf 10} & {\bf 11} &  15  & 16  & 17 \\[0.5ex]
e_{6} & {\bf 10} & {\bf 11} & {\bf 12} & 21 & 22 & 23 \\[0.5ex]
e_{7} & {\bf 11} & {\bf 12} & {\bf 13} &  24 & 25 & 26 \\[0.5ex]
e_{8} & 6 & 7 & 8 &  {\bf 12}  & {\bf 13}  & {\bf 14}  \\[0.5ex]
e_{9} &  {\bf 17} & {\bf 18} &  {\bf 19}  & 23  & 24 & 25\\[0.5ex]
e_{10}  & {\bf 18} & {\bf 19} & {\bf 20} & 29 & 30 & 31\\[0.5ex]
e_{11} & 1 & 2 &  {\bf 19} & {\bf 20} & {\bf 21} & 32 \\[0.5ex]
e_{12} & 14 & 15 & 16 &  {\bf 20}  & {\bf 21}  & {\bf 22}\\[0.5ex]
e_{13} & 1 & {\bf 25} & {\bf 26} & {\bf 27} & 31 & 32 \\[0.5ex]
e_{14} & 5 & 6 & 7 &  {\bf 26}  & {\bf 27}  & {\bf 28}\\[0.5ex]
e_{15} & 8 & 9 & 10 &  {\bf 27} & {\bf 28} & {\bf 29} \\[0.5ex]
e_{16} & 22 & 23 & 24 &  {\bf 28}  & {\bf 29}  & {\bf 30}
\end{array} $ \\[0.5ex]
\end{example}
and, secondly, the hyperedges of $\cH = (V, \{e_1, e_2, \ldots, e_{17} \})$ with $EI(\cH) = C_{34}$ follow.

\begin{example}
Now we have one 5-group $G_0 = \{e_1, e_2, e_3, e_4, e_5 \}$ and three 4-groups \\
$G_1= \{e_6, e_7, e_8, e_9 \}$, $G_2= \{e_{10}, e_{11}, e_{12}, e_{13} \}$, $G_3= \{e_{14}, e_{15}, e_{16}, e_{17} \}$, where\\[2ex]
$ \begin{array}{l@{\, = \, \{ \,}c@{, \,}c@{, \,}c@{, \,}c@{, \,}c@{, \,}c@{\, \},}}
e_{1} & {\bf 1} & {\bf 2} & {\bf 3} &  8  & 9 & 10 \\[0.5ex]
e_{2} & {\bf 2} & {\bf 3} & {\bf 4} &   14 & 15 & 16 \\[0.5ex]
e_{3} & {\bf 3} & {\bf 4} & {\bf 5} &   17 & 18 & 19\\[0.5ex]
e_{4} & {\bf 4} & {\bf 5} & {\bf 6} &  31  & 32  & 33  \\[0.5ex]
e_{5} & 1 & 2 & {\bf 5} & {\bf 6} & {\bf 7} &  34  \\[0.5ex]
e_{6} & {\bf 10} & {\bf 11} & {\bf 12} &  16  & 17  & 18 \\[0.5ex]
e_{7} & {\bf 11} & {\bf 12} & {\bf 13} & 22 & 23 & 24 \\[0.5ex]
e_{8} & {\bf 12} & {\bf 13} & {\bf 14} &  25 & 26 & 27 \\[0.5ex]
e_{9} & 7 & 8 & 9 &  {\bf 13}  & {\bf 14}  & {\bf 15}  \\[0.5ex]
e_{10} &  {\bf 18} & {\bf 19} &  {\bf 20}  & 24  & 25 & 26\\[0.5ex]
e_{11}  & {\bf 19} & {\bf 20} & {\bf 21} & 30 & 31 & 32\\[0.5ex]
e_{12} & 1 & {\bf 20} & {\bf 21} & {\bf 22} & 33 & 34 \\[0.5ex]
e_{13} & 15 & 16 & 17 &  {\bf 21}  & {\bf 22}  & {\bf 23}\\[0.5ex]
e_{14} & {\bf 26} & {\bf 27} & {\bf 28} & 32 & 33 & 34 \\[0.5ex]
e_{15} & 6 & 7 & 8 &  {\bf 27}  & {\bf 28}  & {\bf 29}\\[0.5ex]
e_{16} & 9 & 10 & 11 &  {\bf 28} & {\bf 29} & {\bf 30} \\[0.5ex]
e_{17} & 23 & 24 & 25 &  {\bf 29}  & {\bf 30}  & {\bf 31}
\end{array} $ \\[0.5ex]
\end{example}
the remaining cases $n=36$ and $n=38$ can be handled analogously. \hqed

Together with the short note at the end of B5 (cf. 2.2.1.1, Step 3) the following remark justifies the need for the different constructions for $k=4$ and $k \neq 4$.

\begin{remark}
The construction from the proof of Lemma \ref{Lk4} does not work for all even $n \ge 24$.
\end{remark}

\begin{proof}
To give a simple counterexample, take $n=24$ and assume that the construction of the proof of Lemma \ref{Lk4} does work. It follows $l=0$, therefore we only have 4-groups of hyperedges. We look at $e_2$ and $e_7$, which are from $G_0$ and $G_1$. Then we have $\tj = 0$, $\tj' = 1$, $j = 1$, $j'=5$, $e_2 = e_{j+1}$ and $e_7 = e_{j'+2}$. Hence, we would have to use the construction from (III) of the proof and with $x= 8 \tj$ we obtain  for the hyperedges
$e_2 = \{ 2,3,4,13,14,15 \}$ and $e_7 = \{ 11,12,13,24,25,26 \} =  \{ 11,12,13,24,1,2 \}$ (modulo 24).
This results in the chord $e_2 \, \cap \, e_7 = \{ 2,13 \}$, which is a contradiction.
\hqed
\end{proof}

\begin{remark}
For $n=8 k$, the constructions from the proofs of Lemma \ref{Leven}  and Lemma \ref{Lk4} do not work if $k \le 2$.
\end{remark}

\begin{proof}
If $k \in \{ 0,1 \}$, the remark is trivially true; so let $k=2$. Obviously, we have $l=0$. Since 16 is an integral multiple of 8, we have only two 4-groups of hyperedges in $\cH$. It suffices to consider the hyperedges constructed in the part (III).

First, we use the proof of Lemma \ref{Leven} and consider the hyperedge $e_2$. Then $j=1$, $\tj = 0$ and
$e_2 = e_ {j+1} = \{ 2,3,4,16,17,18 \} = \{ 2,3,4,16,1,2 \}$ (modulo 16). This leads to $e_1 \cap e_2 = \{ 1,2,3 \}$ which results in a chord contradicting $EI( \cH ) = C_n$.

Secondly, concerning the construction in the proof of Lemma \ref{Lk4}, we look at the hyperedge $e_3$. Again, we get $j=1$ and $\tj = 0$. Thus $e_3 = e_ {j+2} = \{ 3,4,5, 16,17,18 \} = \{  3,4,5,16,1,2 \}$ (modulo 16). We obtain the same chord as in the previous case, namely $e_1 \cap e_3 = \{ 1,2,3 \}$, a contradiction.
\hqed
\end{proof}

Finally, to complete the proof of Theorem \ref{Teven}, we have to  investigate  the odd cardinalities $n = | V |$.

\medskip

\subsubsection{Proof of Lemma 3}
\label{Sec_L3}
$\/$

\bigskip

Note that  the present case ($n$ odd) is  simple and therefore there is no necessity to discuss Step 1, Step 2 and Step 3 separately.

\smallskip

Let $ n = 8 k + l + 1$, where $k \ge 3$, $l \in \{ 0,2,4,6 \}$; $\cH' = ( V', \cE')$ with
$EI(\cH') = C_{n-1}$, $V' = \{ 1,2, \ldots, n-1 \}$ and $\cE' = \{ e_1', e_2', \ldots, e_{\frac{n-1}{2}}' \} $. Since $n-1$ is even, we can assume that $\cH'$ (including the numbering of the vertices and the hyperedges) is constructed according to the proofs of Lemma \ref{Leven}
 and Lemma \ref{Lk4} (of course with the (only) modification that now we have $n-1$ vertices instead of $n$ in the lemmas).

Owing to the construction, the vertex 3 and the vertex 4 is contained exclusively in the hyperedges $e_1, e_2, e_3$ and $e_2, e_3, e_4$, respectively; namely in the first 3-sections of these hyperedges.
We obtain $\cH = (V, \cE)$ with $EI(\cH) = C_n$ and $V= \{ 1,2, \ldots, n \}$ from $\cH' = ( V', \cE')$ % with $EI(\cH') = C_{n-1}$
as follows.

Looking at the cycle $C_{n-1}$, we add
a new vertex $n$ "between" the vertices 3 and 4 in this cycle, i.e. in $C_{n-1}=EI(\cH')$, and get $C_n=EI(\cH)$ by the following  construction of $\cH=(V, \cE)$.

\medskip
$ V:= V' \, \cup \, \{ n \}$,

\medskip
$ \cE := \{ e_1, e_2, \ldots, e_{\frac{n+1}{2}} \}$, where

\bigskip
%\hspace*{5cm}
$e_i = \left\{ \begin{array}{l@{\;,\quad}l}
 e_i' & i = 1,4,5,6, \ldots, \frac{n-1}{2} \\[0.8ex]
(e_2' \, \setminus \, \{ 4 \}) \, \cup \, \{ n \} & i = 2 \\[0.8ex]
(e_3' \, \setminus \, \{ 3 \}) \, \cup \, \{ n \} & i = 3  \\[0.8ex]
\{ 3, n, 4 \} & i = \frac{n+1}{2} \, .
 \end{array} \right.$

\bigskip
Instead of the edge $ e_2 \, \cap \, e_3 = \{ 3,4 \} \in \cE(EI(\cH'))$,  in $EI(\cH)$ we have the edges
 $ e_2 \, \cap \, e_{\frac{n+1}{2}} = \{ 3,n \}$ and
 $ e_3 \, \cap \, e_{\frac{n+1}{2}} = \{ n,4 \}$.

 \smallskip
In addition to these non-empty intersections of  $e_{\frac{n+1}{2}}$ with $e_2$ and $e_3$,  in $\cH$ we have only two further non-empty intersections with other hyperedges, namely
 $ e_1 \, \cap \, e_{\frac{n+1}{2}} = \{ 3  \}$ and  $ e_4 \, \cap \, e_{\frac{n+1}{2}} = \{ 4  \}$.
 Thus, the hyperedge  $e_{\frac{n+1}{2}}$ does not induce any chord in $C_n$. The same holds for the (modified) hyperedges $e_2$ and $e_3$. Hence $EI(\cH) = C_n$, where $C_n$ is the vertex sequence
 $(\ldots, n-2, n-1, 1,2,3, n, 4,5,6, \ldots ).$ \hqed

\bigskip
Thereby, the proof of Theorem 2 is complete.

\subsection{Concluding remarks}

Looking at the laborious constructions in the proof of Theorem 2, as a next step it seems to be sensible to investigate the following specialized version of Problem 1 (see Section 1).

\medskip
{\bf Problem 2.}
Let $k \ge 3$, $n_0 \in \N^+ $ and $n \ge n_0$. What is the minimum cardinality $| \cE |$ of the edge set of a $3k$-uniform hypergraph $\cH =(V, \cE)$ with $EI( \cH) = C_n$?

\medskip

Note that Corollary 1 and Theorem \ref{Teven} gives the solution for the 3-uniform ($k=1$) and the 6-uniform ($k=2$) case, respectively.

\smallskip
In order to motivate the concentration on $3k$-uniform hypergraphs, we consider  Theorem \ref{Tmin} and Theorem \ref{Teven}.  For $k=2$, we verified that to construct  edge-minimal 6-uniform hypergraphs $\cH$ with $EI(\cH) = C_n$, the hyperedges of $\cH$ can be composed from certain  3-sections of $C_n$. The combination of 3-sections results in hyperedges of cardinalities being integral multiples of 3.
%This approach seems to be a likely one also for hyperedges of larger cardinality $r$, which leads to $ r = 3 k$ .
Following this likely approach also for hyperedges of larger cardinality $r$, this leads to $ r = 3 k$ ($k \ge 3$).

%We conjecture that the solution of the problem is difficult for $k \ge 3$.


\begin{thebibliography}{}	
\bibitem{GST5} C. Berge, \textit{Graphs and Hypergraphs}, North Holland, Amsterdam, 1973.	
\bibitem{Bie} Th. Biedl, M. Stern, \textit{On edge-intersection graphs of $k$-bend paths
in grids}, Discr.  Math. and Theor. Comp. Sci. \textbf{12(1)} (2010) 1--12.
\bibitem{Cam} K. Cameron, S. Chaplick, C.T. Hoang, \textit{Edge intersection graphs of $L$-shaped paths in grids}, Discr. Appl. Math. \textbf{210} (2016) 185--194.
%\bibitem{GST2016}  C. Garske, M. Sonntag, H.-M. Teichert, \textit{Niche hypergraphs}, Discussiones Mathematicae Graph Theory \textbf{36} (2016) 819--832.
\bibitem{Gol} M.C. Golumbic, R.E. Jamison, \textit{The Edge Intersection Graphs
of Paths in a Tree}, J. Comb. Theory, Series B \textbf{38} (1985) 8--22.
\bibitem{Nai} R.N. Naik, S.B. Rao, S.S. Shrikhande, N.M. Singhi, \textit{Intersection Graphs of $k$-uniform Linear Hypergraphs}, Europ. J. Combinatorics \textbf{3} (1982) 159--172.
%\bibitem{GST33} J. Park, Y. Sano, \textit{The double competition hypergraph of a digraph}, Discr. Appl. Math. \textbf{195} (2015) 110 --113.
%\bibitem{ROB} G. Robers, \textit{EI-Hypergraphen und deren Eigenschaften}, Bachelor Thesis, University of L\"ubeck, 2017.
\bibitem{Sku} P.V. Skums, S.V. Suzdal, R.I. Tyshkevich, \textit{Edge intersection graphs of linear 3-uniform hypergraphs}, Discr. Math. \textbf{309} (2009) 3500--3517.
%\bibitem{GST40} M. Sonntag, H.-M. Teichert, \textit{Competition hypergraphs}, Discr. Appl. Math.  \textbf{143} (2004) 324--329.
\bibitem{ST2019} M. Sonntag, H.-M. Teichert, \textit{Edge intersection hypergraphs -- a new hypergraph concept}, arXiv:1901.06292[math.CO] (2019) 1--15. (submitted for publication)
\end{thebibliography}
\end{document}